\documentclass[a4paper,10pt]{article}

\usepackage[colorlinks=true, allcolors=blue]{hyperref}
\usepackage{authblk}
\usepackage{bbm}
\usepackage{overpic}
\usepackage{rotating}
\usepackage{xcolor}
 \usepackage{relsize}

\usepackage[utf8]{inputenc}
\usepackage[english]{babel}
\usepackage{csquotes}
\usepackage{geometry}
\usepackage{amsfonts}
\usepackage{amsmath}
\usepackage{amsthm, thmtools}
\usepackage{amssymb}
\usepackage{xcolor}

\usepackage{caption}
\usepackage{subcaption}

\DeclareGraphicsExtensions{.png,.pdf,.jpg}

\numberwithin{equation}{section}

\newcommand{\TU}[1]{\textup{#1}}

\newtheorem{theorem}{Theorem}[section]
\newtheorem{lemma}[theorem]{Lemma}
\newtheorem{proposition}[theorem]{Proposition}
\newtheorem{corollary}[theorem]{Corollary}

\newtheorem{definition}[theorem]{Definition}
\newtheorem{remark}[theorem]{Remark}

\begin{document}
\title{A singular perturbation analysis for the Brusselator}

\author[1,2]{Maximilian Engel\thanks{maximilian.engel@fu-berlin.de}}

\author[1,3]{Guillermo Olic\'on-M\'endez\thanks{g.olicon.mendez@fu-berlin.de}}

\affil[1]{Freie Universit\"at Berlin, Institut f\"ur Mathematik und Informatik, 14195 Berlin, Germany}

\affil[2]{University of Amsterdam, KdV Institute for Mathetmatics, 1098 XG Amsterdam, Netherlands.}

\affil[3]{Universit\"at Potsdam, Institut f\"ur Mathematik, 14476 Potsdam, Germany}

\maketitle

\abstract{In this work we study the Brusselator -- a prototypical model for chemical oscillations-- under the assumption that the bifurcation parameter is of order $O(1/\epsilon)$ for positive $\epsilon\ll 1$. The dynamics of this mathematical model exhibits a time scale separation visible via fast and slow regimes along its unique attracting limit cycle. Noticeably this limit cycle accumulates at infinity as $\epsilon\rightarrow 0$, so that in polar coordinates $(\theta,r)$, and by doing a further change of variable $r\mapsto r^{-1}$, we analyse the dynamics near the \textit{line at infinity}, corresponding to the set $\{r=0\}$. This object becomes a nonhyperbolic invariant manifold for which we use a desingularising rescaling, in order to study the closeby dynamics. Further use of geometric singular perturbation techniques allows us to give a decomposition of the Brusselator limit cycle in terms of four different fully quantified time scales.

$ \ $



\section{Introduction}
The Brusselator model, introduced originally by Ilya Prigogine and René Lefever in \cite{Prigogine68}, has become a milestone in the study of nonlinear chemical dynamics, for its mathematical simplicity and its dynamical richness -- the model admits the minimal conditions for \textit{dissipative structures} to emerge and thus remaining far away from equilibrium. Specifically, under suitable assumptions the Brusselator can be modeled as a 2-dimensional system of ordinary differential equations (ODEs) with a cubic term (associated to a trimolecular reaction step), allowing for self-sustained oscillations to occur \cite[Chapter 7]{NicolisPrigogine}. While earlier and contemporary models of chemical reactions exhibit oscillations, namely those of Turing \cite{Turing52}, Lotka \cite{Lotka20}, and Sel'kov \cite{Selkov68}, 
the Brusselator stands out for being structurally stable and admitting a global attractor for all values of the parameters, see \cite{Erneaux18} for a historical context.

Furthermore, when the Brusselator system is allowed to diffuse (by following \textit{Fick's law}, for example), the Brusselator's spatiotemporal dynamics can be modeled as a reaction-diffusion equation, which exhibits further succesive nonhomogeneous bifurcations \cite{NicolisPrigogine}. Notably, under a suitable combination of the system's parameters, the Brusselator exhibits \textit{Turing bifurcations} \cite{Turing52}, which induce spatiotemporal patterns with different geometrical features \cite{Pena01}. Moreover, it has been proven that under the influence of an external periodic forcing, the system may exhibit chaotic dynamics \cite{Young13}.

In recent years, the development of a bifurcation theory for \textit{random dynamical systems} has permeated also in the field of chemical oscillators. This approach addresses the convergence or separation of nearby initial conditions, evolving under the influence of a common noisy perturbation.
For instance, in \cite{LiZhang19} a so-called \textit{dynamical bifucation} occurs due to a stochastic pitchfork bifurcation of the equilibrium. When the noisy term is included in the Brusselator as a stochastic variation of the bifurcation parameter, numerical explorations conjectured that \textit{parametric noise destroys the Hopf bifurcation} in the system, featuring a unique global attracting random equilibrium, regardless of the values of the parameters \cite{ArnoldBrus}. While this conjecture, to the best of our knowledge, remains still open, in \cite{EngelOlicon23} further numerical explorations exhibit arbitrarily long time windows where this stochastic Brusselator shows noise-induced finite-time instabilities, when the bifurcation parameter is large. 

Already in the absence of noisy perturbations, the Brusselator exhibits a time scale separation precisely when the bifurcation parameter is large. This setting has already been studied in \cite{Nicolis71} from a qualitative point of view and in \cite{Boa76} using asymptotic expansion techniques. 
In this paper we study the Brusselator in its classical ODE formulation under the scope of \textit{geometric singular perturbation theory} (GSPT), precisely under the assumption that the bifurcation parameter is of order $\mathcal{O}(1/\epsilon)$ for positive $\epsilon\ll 1$. The time scale separation of its dynamics becomes readily evident when expressed in the right coordinate frame (see \cite{Nicolis71, LiHouShen16} or Section~\ref{SUBSEC:time_scale_sep} below). 

We show that the Brusselator exhibits \textit{relaxation oscillations}, given by a globally attracting limit cycle, which switches between fast and slow regimes. While the existence of a unique attracting limit cycle can be deduced by simple phase plane analysis or via a transformation which turns the system into a Lienard equation \cite{Ponzo78}, the advantage of using GSPT in the analysis is that we recover important qualitative and quantitative features of the global dynamics like:
\begin{itemize}
    \item the different curves that approximate the limit cycle,
    \item the order of approximation (in terms of the small parameter $\epsilon$), and
    \item the distinct time scales spent near such branches.
\end{itemize}

The detailed geometrical description of the Brusselator presented in this work aims to constitute a first approach to the study of the multiple time scale dynamics of the system, even for more complex variations where, like in the examples mentioned earlier, noisy perturbations or space-time dynamics are involved.

\subsubsection*{GSPT in chemical oscillators}
GSPT addresses the challenges of understanding the dynamics near singular objects in ODEs by employing a combination of analytical and geometrical tools to uncover the underlying structure of the system's solutions. It has proven to be a powerful tool in the study of oscillators, particularly of those of (bio)chemical nature. For instance, in \cite{Szmolyan09}, GSPT methods were used to analyse an autocatalator model, whose analytic expression is very similar to the Brusselator's. In \cite{Kosiuk11}, a multiple time scale analysis is performed on a glycolytic oscillator.

In this work we follow closely some of the techniques used in \cite{Szmolyan09} and \cite{Kosiuk11}, in order to perform a multiple time scale analysis of the Brusselator. The main idea is to construct a \textit{Poincar\'e map} as a composition of different transition maps which closely track the different branches of the limit cycle.
In order to achieve this goal, firstly it is easily shown that the system admits an unbounded \textit{critical manifold} which is partially hyperbolic, and where a slow behaviour dominates. This invariant object is tracked by the limit cycles, until the limit cycle jumps into the fast regime. While typically this transition is induced by a loss of hyperbolicity on the critical manifold, this is not the case for the Brusselator. Inspired by \cite{Kuehn14}, where analysis around unbounded critical manifolds is conducted,
we transform the system into polar coordinates of the form
\begin{equation}
\label{eq:first_polar}
    x=\frac{\cos\theta}{r}, \qquad y=\frac{\sin\theta}{r},
\end{equation}
where we are interested in the dynamics near $r=0$. The resulting equations become a \textit{slow-fast system in nonstandard form} \cite{Wechselberger2020}. After an appropriate time rescaling, the line $\{r=0\}$ (corresponding to the \textit{line at infinity} in the original system) consists of nonhyperbolic equilibria. The essential dynamics as $\epsilon\rightarrow 0$ is ``hidden'' near this curve, for which a further desingularising rescaling is necessary in order to reveal it.

While most of the relevant dynamics can be analysed, and corresponding transition maps can be defined, a single nonhyperbolic equilibrium persists. We further perform a \textit{blow-up transformation} in order to desingularise it and to construct a final transition map. The blow-up method has become a standard tool in order to analyse the dynamics near nonhyperbolic objects, going back to works by Dumortier \cite{Dumortier1978, Dumortier1993} and introduced to fast-slow systems by Dumortier and Roussarie \cite{DumortierRoussarie1996}. The method has yielded many additional results in recent years, for instance, in planar fast-slow ODEs \cite{DeMaesschalckDumortier2005, DeMaesschalckDumortier2010, Krupa01, KrupaSzmolyan2001b}, fast-slow maps \cite{ EngelKuehn2019, Engelcanards, JelbartKuehn, NippStoffer} or also application-oriented problems \cite{DeMaesschalckWechselberger, Szmolyan09}; see \cite{JardonKuehn18} for a survey on the blow-up method and its applications in GSPT.

\subsubsection*{Structure of the paper}
In Section~\ref{SEC:Brusselator}, we introduce the dynamics of the Brusselator and its slow-fast structure given by a linear change of coordinates. We state the main theorem of this work in this coordinate frame. We further perform a first analysis near infinity via the change of variables \eqref{eq:first_polar}. In Section~\ref{SECTION:blowup}, we do a rescaling of the system which desingularises the line $\{r=0\}$, and permits the construction of a singular cycle towards which the limit cycle approaches as $\epsilon\rightarrow 0$. In Section~\ref{SECTION:poincareMap}, we construct a Poincaré map as a composition of four transition maps between transverse sections to the different branches of the singular cycle. This map is a contraction, implying the existence of a unique fixed point which is attracting. In Section~\ref{SEC:MainProof}, we translate our results to the original coordinates. Finally, in Section~\ref{SEC:conclusion} we bring some conclusions and and outlook for future work. 

\subsubsection*{Notation and terminology}
In this work we deal with systems of ODEs of the form
\begin{equation}   \label{eq:general_form}
    \dot{X}=F_0(X)+\epsilon F_1(X,\epsilon),
\end{equation}
where $X\in \mathbb{R}^2$, $F_0,F_1$ are $C^{\infty}$ vector fields for each value of the perturbation parameter $\epsilon\in(0,\epsilon_0)$, where $\epsilon_0\ll1$, and $\dot{\;}=d/dt$ indicates differentiation with respect to the time variable $t\in\mathbb{R}$, being mostly concerned with solutions of the ODE for $t\in\mathbb{R}_+:=[0,\infty)$. Furthermore, the vector field $F_0$ is assumed to
admit a factorization of the form
\begin{equation}
    \label{eq:factorization}
        F_0(X)=N(X)f(X),
\end{equation}
where $N:\mathbb{R}^2\rightarrow\mathbb{R}^2$ is a nonvanishing smooth vector field and $f:\mathbb{R}^2\rightarrow\mathbb{R}$ is a smooth scalar field. We refer to a system of such type as a \textit{slow-fast system in general form}. For an excellent introduction to GSPT in the general setting, see \cite{Wechselberger2020}. 
The dynamics of~\eqref{eq:general_form} can be understood as a singular perturbation of the so-called \textit{layer problem}
\begin{equation}
\label{eq:layer_intro}
    \dot{X}=F_0(X).
\end{equation}
In all our examples, there is a set $\mathcal{S}_0\subset\mathbb{R}^2$ where the layer problem vanishes, i.e
\[
    \mathcal{S}_0=\left\{ X\in\mathbb{R}^2 : F_0(X)=0 \right\} 
    = \left\{ X\in\mathbb{R}^2 : f(X)=0 \right\},
\]
which is a union of curves whose intersection set consists of isolated points.

By performing the time rescaling to the slow variable $\tau:=\epsilon t$, we obtain the equivalent ODE system
\begin{equation}
\label{eq:reduced_intro}
    X'= \frac{1}{\epsilon}N(X) f(X) +
    F_1(X,\epsilon),
\end{equation}
where $\cdot ':=d/d\tau$. The singular limit $\epsilon\rightarrow0$ for this equation is well defined only when the initial condition lies in $\mathcal{S}_0$ and $F_1(X,0)$ maps to the tangent space $T_X\mathcal{S}_0$. We refer to such a restricted system as the \textit{reduced problem}.

Since $\mathcal{S}_0$ consists entirely of equilibria, any compact submanifold of $\mathcal{S}_0$ is invariant for the layer problem \eqref{eq:layer_intro}. It is natural to ask whether it persists in the $C^r$-topology as an invariant manifold $\mathcal{S}_\epsilon$ for \eqref{eq:general_form} when $\epsilon>0$. In \cite{Fenichel71}, Neil Fenichel gave a positive answer in a general setting, under the assumption that the invariant manifold is \textit{normally hyperbolic}. In our specific setting, a compact connected submanifold $M\subset \mathcal{S}_0$ is normally hyperbolic, if for every $z\in M$ the spectrum of the Jacobian matrix $DF_0(z)$ has a nonzero eigenvalue\footnote{Since $\mathcal{S}_0$ is invariant and the state space is $\mathbb{R}^2$, the second eigenvalue is necessarily $0$.}.
In the specific case of $N=(1,0)^{\top}$ in the factorisation \eqref{eq:factorization}, we can set $X=(x,y)$ such that the fast and slow systems read respectively as
\[
    \begin{array}{rcl}
    \dot{x}&=&f(x,y,\epsilon)\\
    \dot{y}&=&\epsilon g(x,y,\epsilon)
    \end{array},
    \qquad
     \begin{array}{rcl}
    \epsilon x'&=&f(x,y,\epsilon) \\
    y'&=&g(x,y,\epsilon)
    \end{array},
\]
for some smooth scalar fields $f,g:\mathbb{R}^2\times\mathbb{R}_+\rightarrow\mathbb{R}$. The systems written in this specific form constitute a \textit{slow-fast system in standard form}, and whenever this is not the case we refer to them as in \textit{nonstandard form}. For a comprehensive overview of GSPT in standard form, see \cite{Kuehn15}. On the one hand, in the standard case, the layer problem becomes simply
\[
\begin{array}{rcl}
    \dot{x} &=& f(x,y,0)\\
    \dot{y} &=& 0,
\end{array}
\]
which is a family of one-dimensional ODEs, parametrized by the initial condition $y=y_0$. On the other hand, the reduced problem becomes an algebraic-differential equation of the form
\[
\begin{array}{rcl}
    0&=&f(x,y,0)\\
    y'&=&g(x,y,0),
\end{array}
\]
where the dynamics is defined on the critical set
\[
    \mathcal{S}_0=\{(x,y) : f(x,y,0)=0\}.
\]
Fenichel's theorems on the smooth persistence of $S_0$ and related results for systems in standard form were originally proved in \cite{Fenichel79}. We refer the reader to \cite[Chapter 2]{Kuehn15} for a concise exposition of general Fenichel theory, even in higher dimensional systems.

\noindent\textit{Asymptotic expansions.} Throughout this work we approximate real/vector valued functions $f(\epsilon)$ as $\epsilon\rightarrow 0$ by an appropriate real/vector valued function $g$. We say that:
\begin{itemize}
    \item $f(\epsilon)=\mathcal{O}\left(g(\epsilon) \right)$ as $\epsilon\rightarrow 0$, if there exists $K>0$ and $\epsilon_0>0$ such that
    \[\Vert f(\epsilon) \Vert\leq K \Vert g(\epsilon)\Vert\] for all $\epsilon \in[0,\epsilon_0]$,
    \item $f(\epsilon)=\Omega\left(g(\epsilon) \right)$ as $\epsilon\rightarrow 0$, if there exists $K>0$ and $\epsilon_0>0$ such that
    \[\Vert f(\epsilon) \Vert\geq K \Vert g(\epsilon)\Vert\] for all $\epsilon \in[0,\epsilon_0]$,
     \item $f(\epsilon)=\Theta\left(g(\epsilon) \right)$ as $\epsilon\rightarrow 0$, if $f(\epsilon)=\mathcal{O}(g(\epsilon))$ and $f(\epsilon)=\Omega(g(\epsilon))$ as $\epsilon\rightarrow0$.
    \item $f(\epsilon)=o\left(g(\epsilon) \right)$, if 
    \[
        \lim_{\epsilon\rightarrow 0} \frac{\Vert f(\epsilon)\Vert}{\Vert g(\epsilon)\Vert}=0,
    \]
\end{itemize}
In the following we omit the reference to $\epsilon\rightarrow 0$, as it should be clear from the context.  

We are interested in the time scales of the dynamics in different regions of the state space, according to the following definition.
\begin{definition}
\label{DEF:timescale}
We say that the \textit{time scale} of a dynamical system $\varphi_t$ depending on a small parameter $\epsilon$ near a curve $\mathcal{C}=\{\sigma(t): t\in[a,b]\}$ is $\mathcal{O}(g(\epsilon))$ if for every $\delta>0$ sufficiently small there are transverse sections $\Sigma_0$ and $\Sigma_{hit}$ through $\sigma(a+\delta)$ and $\sigma(b-\delta)$, respectively,  such that
\begin{enumerate}
    \item the length of $\Sigma_{hit}$ is at most $\delta$,
    \item for every initial condition $x_0\in \Sigma_0$, the hitting time
    \[
    \tau_+\equiv\tau_+(x_0):=\min\{t\geq 0 : \varphi_t(x_0)\in \Sigma_{hit}\}
\]
is finite, and
\item $\tau_+$ satisfies that $\tau_+=\mathcal{O}(g(\epsilon))$. 
\end{enumerate}
\end{definition}
\noindent Analogous definitions follow by replacing $\mathcal{O}\left(g(\epsilon) \right)$ with $\Omega(g(\epsilon))$, $\Theta(g(\epsilon))$, or $o(g(\epsilon))$.

\section{Slow-fast structure for the Brusselator}
\label{SEC:Brusselator}
We consider the Brusselator chemical reaction network
\begin{equation}\label{eq:BRUSSsystem}
\begin{array}{lrcl}
R_1: &	A & \rightarrow& X \\
R_2: &	B+X & \rightarrow& Y+D\\
R_3: &	2X+Y & \rightarrow& 3X\\
R_4: &	X & \rightarrow& E,
\end{array}
\end{equation}
where we assume that the reaction rate of each reaction channel $R_i$ ($i=1,2,3,4$) obeys the \textit{law of mass actions}. We further assume that all chemical species in \eqref{eq:BRUSSsystem} scale adequately with the volume of the domain where the reactions occur, so that the evolution of their concentrations are properly modeled via a system of ordinary differential equations (ODEs), named \textit{reaction rate equations} \cite{Winkelmann2020}. The concentration of species $A$ and $B$ are assumed to be constant parameters. Under these assumptions, the dynamics of \eqref{eq:BRUSSsystem} is given by
\begin{equation}
\label{eq:BrusDet}
\begin{array}{rcl}
    X'&=&a-(1+b)X+X^2Y, \\
    Y'&=&bX-X^2Y,
\end{array}
\end{equation}
where $X(\tau),Y(\tau)$ are the densities of chemical species $X,Y$ at time $\tau$, and $a,b>0$ are the concentrations of $A$ and $B$, respectively. As before, $':=d/d\tau$. Because of its chemical interpretation, the state space is the set $\mathbb{R}_+^2:=\{(X,Y) : X,Y\geq 0\}$.

\subsection{Hopf bifurcation}
The system \eqref{eq:BrusDet} has a unique equilibrium point given by $(X,Y)=(a,b/a)$. The Jacobian matrix on the equilibrium is given by
\[
    \left[ 
    \begin{array}{cc}
        b-1 & a^2 \\
        -b & -a^2
    \end{array}
    \right].
\]
Since its stability is given by its trace, it follows that the equilibrium is (locally) asymptotically stable if and only if $b<1+a^2$. The system exhibits a \textit{supercritical Hopf bifurcation} at $b_{crit}=1+a^2$, where the equilibrium becomes unstable whenever $b>b_{crit}$, and an attracting limit cycle emerges (see for instance \cite[Theorem 3.3]{kuznetsov04}). We refer to Figure~\ref{fig:limitcycle_growing} for a qualitative depiction of the limit cycle for growing values of $b$.
\begin{figure}[ht]
     
      \begin{subfigure}{.45\linewidth}
      \centering
      \begin{overpic}[width=\linewidth]{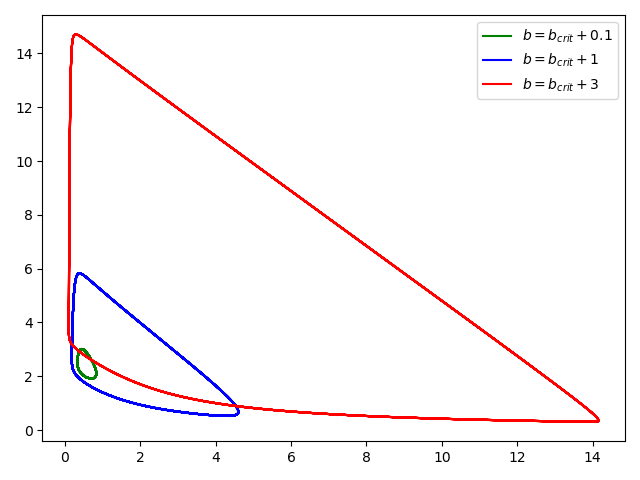}
      \put(50,-2){$X$}
      \put(-3,40){$Y$}
      \end{overpic}
      \caption{}
      \end{subfigure}
      \begin{subfigure}{.45\linewidth}
      \centering
      \begin{overpic}[width=\linewidth]{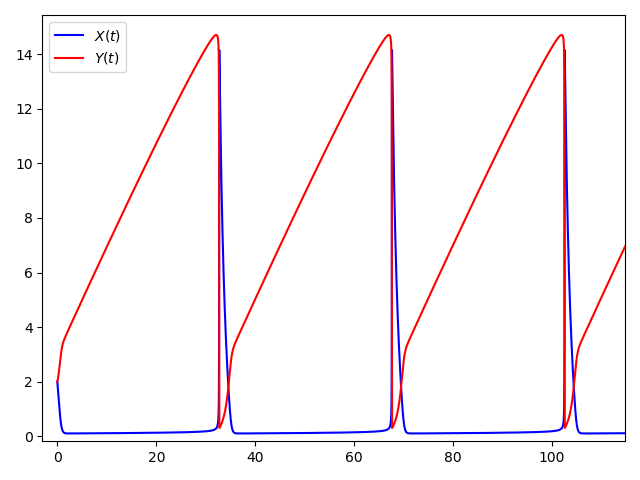}
      \put(50,-2){$t$}
      \end{overpic}
      \caption{}
      \end{subfigure}

       \hfill
       \caption[Limit cycles of the Brusselator for different parameters]{In (a), three examples of the limit cycle $\gamma_b$ of \eqref{eq:BrusDet} for different values of the parameter $b$ are portrayed, while fixing $a=0.5$. Observe that the limit cycle grows in size as $b$ grows. Moreover, as we see later, the cycle $\gamma_b$ diverges entirely in the sense that $\inf_{z\in\gamma_b} \Vert z\Vert\rightarrow \infty$ as $b\rightarrow\infty$. In (b), the time series of \eqref{eq:BrusDet} are presented for $a=0.5$ and $b=b_{crit}+3$. Note that $Y(t)$ exhibits a big drop as the transition from a slow regime into a fast regime occurs, which seems to happen almost instantaneously. Similarly, as $X(t)$ behaves in antiphase to $Y(t)$, its behaviour resembles a train of instantaneous pulses.}
        \label{fig:limitcycle_growing}
\end{figure}

\subsection{Time scale separation}
\label{SUBSEC:time_scale_sep}
In the following we are interested in parameter values such that $b\gg a$. We fix a value for $a>0$, and assume that $b=a/\epsilon$, where $\epsilon\ll 1$, i.e.~in particular $b>b_{crit}$.
A slow-fast structure in the system \eqref{eq:BrusDet} is unveiled via the change of coordinates used in \cite{Nicolis71, LiHouShen16}
\[
\begin{array}{rcl}
    x&=&Y,\\
    y&=&X+Y.
\end{array}
\]
In this new coordinate system, the dynamics are given by the ODEs
\begin{equation}
    \label{eq:slowDynamics}
    \begin{array}{rcl}
        \epsilon x'&=&a(y-x)-\epsilon x(y-x)^2,\\
        y'&=& a-(y-x).
    \end{array}
\end{equation}
Equivalently, by a rescaling of the time variable $t=\tau/\epsilon$, we obtain the system
\begin{equation}
    \label{eq:fastDynamics}
    \begin{array}{rcl}
        \dot{x}&=&a(y-x)-\epsilon x(y-x)^2,\\
        \dot{y}&=&\epsilon\left[ a-(y-x)\right],
    \end{array}
\end{equation}
where $\dot{\;} := d/dt$.

Equations \eqref{eq:slowDynamics} and \eqref{eq:fastDynamics} form a \textit{fast-slow system} in standard form, where the slow dynamics is captured by system (\ref{eq:slowDynamics}) and the fast dynamics by \eqref{eq:fastDynamics}, respectively. The reduced problem, by setting $\epsilon=0$ in \eqref{eq:slowDynamics}, reads as
\begin{equation}
    \label{eq:reduced_1}
        \begin{array}{rcl}
            0 &=& a(y-x), \\
            y' &=& a-(y-x),
        \end{array}
\end{equation}
while the \textit{layer problem} is obtained by replacing $\epsilon=0$ in \eqref{eq:fastDynamics}, obtaining
\begin{equation}
    \label{eq:layer}
        \begin{array}{rcl}
            \dot{x} &=& a(y-x), \\
            \dot{y} &=& 0.
        \end{array}
\end{equation}
Equation~\eqref{eq:reduced_1} yields a \emph{reduced flow} on the \emph{critical manifold} 
\[
    S_0:=\left\{ (x,y)\in \mathbb{R}^2_+ : x=y \right\}.
\]
Observe that $S_0$ is a set of points which are equilibria for \eqref{eq:layer}. 
It can be readily observed that $S_0$ is an attracting \emph{normally hyperbolic invariant manifold}. Indeed, let $h(x,y)=(a(y-x),0)^\intercal$ so that
\[
    Dh(z)=\left[ \begin{array}{cc}
                        -a & a\\
                        0 & 0
                    \end{array}\right].
\]
It is then straightforward to see that there is a nontrivial eigenvalue, namely $\lambda_1=-a<0$, at any $z\in S_0$. Notice that since the system is given in standard form, this is equivalent to the condition $\partial_x f(x,x)<0$, where $x'=f(x,y)$.

The solutions of \eqref{eq:reduced_1} can be calculated explicitly: given $x_0\in \mathbb{R}_+$, and an initial condition $(x_0,x_0)\in S_0$, we have that
\[
    x(t)=y(t)=x_0+at.
\]
By Fenichel's theorem, see for instance \cite[Theorem 2.4.2]{Kuehn15}, for each compact interval $[a,b]\subset \mathbb{R}_+$ there is $\epsilon_0$ sufficiently small such that for every $\epsilon\in[0,\epsilon_0]$ there exists a locally invariant (slow) manifold 
\[
    S_{\epsilon}=\{(x,y)\in\mathbb{R}^2 : y=H(x,\epsilon)\},
\]
where $H$ is a smooth function such that $H(x,0)=x$.

As seen in Figure~\ref{fig:limitcycle_standard}, the limit cycle follows $S_0$ closely for $\epsilon>0$, and eventually takes a turn. In many other situations, this turn is generically due to a loss of hyperbolicity. As calculated above, this is not the case for $S_0$ and it will be a crucial contribution of the paper to understand the behavior of the  limit cycle in this situation.

For simplicity, in the rest of this paper we refer to a curve $\{\sigma_i(t):t\in[a,b]\}$ simply as $\sigma_i$. We state the main result of this work in the following theorem.
\begin{theorem}
    \label{THM:limit cycle_ original coordinates}
        For each $\epsilon>0$ sufficiently small, system \eqref{eq:fastDynamics} admits a unique attracting limit cycle $\gamma^\epsilon$, which exhibits a time scale separation (cf. Definition~\ref{DEF:timescale}) near the
        cycle composed of four curves parameterised by the functions
\begin{align*}
           \sigma_1(\theta)=& \frac{1}{\sqrt{\epsilon}}
           \left( \frac{\cos\theta}{\rho_{\epsilon}\sin\theta}, \frac{1}{\rho_\epsilon} \right), \qquad \textup{ for } \theta\in[\pi/4,\pi/2], 
          \\
            \sigma_2(\theta)=& \frac{1}{\sqrt{\epsilon}}\left( 
            \sqrt{\frac{a\cos\theta}{\sin\theta-\cos\theta}}, \sqrt{\frac{a\sin^2\theta}{\cos\theta(\sin\theta-\cos\theta)}}\right), \quad \textup{ for } \theta\in\left[ \arctan(2),\pi/2 \right],
        \\
            \sigma_3(\theta)=& \frac{1}{\sqrt{\epsilon}}\left( 2\sqrt{a}\cot\theta, 2\sqrt{a} \right), \qquad \textup{ for } \theta\in\left[\pi/4,\arctan(2) \right],
        \\
             \sigma_4(r)=& \frac{1}{\sqrt{\epsilon}}\left( 
          \frac{1}{\sqrt{2}r}, \frac{1}{\sqrt{2}r}
          \right), \qquad \textup{ for } r\in\left[\rho_{\epsilon},\frac{1}{2\sqrt{2a}}\right],
\end{align*}
where $\rho_{\epsilon}=\Theta\left( \epsilon^{3/2} \right)$. The corresponding time scales near each branch are:
\begin{itemize}
    \item near $\sigma_1$, $t=\Theta \left( \epsilon^{3} \right)$, 
    \item near $\sigma_2$, $t=\Theta\left(\epsilon^{-1} \right)$,
    \item near $\sigma_3$, $t=\Theta(1)$, and
    \item near $\sigma_4$, $t=\Theta\left(\epsilon^{-3/2} \right)$.
\end{itemize}
\end{theorem}

\begin{figure}[ht]
     \centering
       \begin{overpic}[width=.5\linewidth]{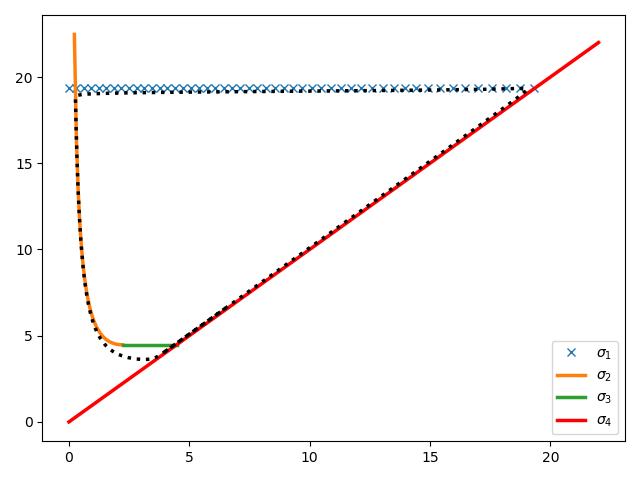}
    \put(50,-2){$x$}
    \put(-4,40){$y$}
    \end{overpic}
       \hfill
       \caption[Limit cycle of the Brusselator]{A numerical approximation of the limit cycle of system \eqref{eq:slowDynamics} for $a=0.5$ and $\epsilon=0.1$ is displayed as black, dotted curve. The approximation of the limit cycle is given by the curves $\sigma_1$, $\sigma_2$, $\sigma_3$, and $\sigma_4$, which are portrayed here. The curve $\sigma_1$ was obtained numerically, as there is no explicit expression of $\rho_\epsilon$.}        
      \label{fig:limitcycle_standard}
\end{figure}
We refer the reader to Section~\ref{SEC:MainProof} for the proof of this theorem which describes the dynamics of the Brusselator as a four-stroke relaxation oscillator, where its behaviour is split in the transitions ultrafast-to-slow-to-fast-to-superslow in each cycle. We call the regime near $\sigma_1$ ultrafast due to the fast behaviour which resembles an ``instantaneous jump of infinite length" as $\epsilon\rightarrow0$, and the regime near $\sigma_4$ superslow  since it is the slowest time scale of the four listed in Theorem~\ref{THM:limit cycle_ original coordinates}.
In Figure~\ref{fig:limitcycle_standard}, it becomes evident that there is a very good agreement between the limit cycle and the curves $\sigma_2$ and $\sigma_4$, while the deviation from $\sigma_3$ is an effect of a transition through a regular fold point, see Subsection~\ref{SEC:PI_23}. Note that  $\sigma_2$ and $\sigma_4$ are parameterisations of the two branches of the nullcline $\dot{x}=0$ for \eqref{eq:fastDynamics}, as already obtained in \cite{LiHouShen16}; $\sigma_2$ satisfies $y=x+\frac{a}{\epsilon x}$ and $\sigma_4$ satisfies $y=x$. Theorem~\ref{THM:limit cycle_ original coordinates} does not give an explicit calculation of $\rho_\epsilon$, for which only a numerical approximation of $\sigma_1$ can be provided in Figure~\ref{fig:limitcycle_standard}.

Just as the numerical solutions indicate (see Figure~\ref{fig:limitcycle_growing}): since each $\sigma_i$ diverges entirely to infinity as $\epsilon\rightarrow0$, then so does the limit cycle of \eqref{eq:slowDynamicsCompact}. However, Theorem~\ref{THM:limit cycle_ original coordinates} suggests that the limit cycle is better observed when rescaling $(\bar{x},\bar{y})=\left(\sqrt{\epsilon}x,\sqrt{\epsilon}y  \right)$, see Figure~\ref{fig:limitcycle_comp}. A reformulation of Theorem~\ref{THM:limit cycle_ original coordinates} in this rescaled coordinate system can be found in Corollary~\ref{CORO:rescaled}. Its advantage is that the curves $\bar{\sigma}_i=\sqrt{\epsilon}\sigma_i$ for $i=2,3,4$ are fixed geometric objects which are independent of $\epsilon$, and, hence, allow us to compare them with the rescaled limit cycle $\bar{\gamma}^{\epsilon}=\sqrt{\epsilon} \gamma^{\epsilon}$.

\begin{figure}[ht]
     \centering  
      \begin{overpic}[width=.5\linewidth]{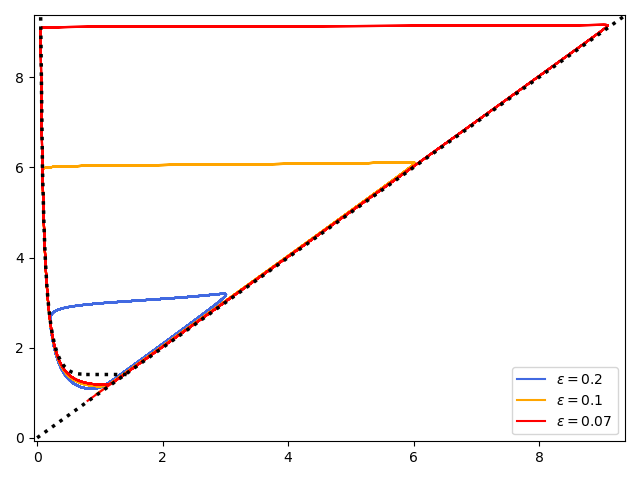}
    \put(50,-2){$\sqrt{\epsilon} x$}
    \put(-5,40){$\sqrt{\epsilon} y$}
    \put(7,57){\scriptsize $\sqrt{\epsilon}\sigma_2$}
    \put(10,19){\scriptsize $\sqrt{\epsilon}\sigma_3$}
    \put(50,35){\scriptsize $\sqrt{\epsilon}\sigma_4$}
    \end{overpic}
       \hfill
       \caption[Variation of Limit cycles]{The profile of the limit cycles in the $(\sqrt{\epsilon}x,\sqrt{\epsilon} y)$-plane, for different values of $\epsilon$. As $\epsilon$ becomes smaller, the examples agree better with the rescaled curves $\sqrt{\epsilon}\sigma_i$, $i=2,3,4$, from Theorem~\ref{THM:limit cycle_ original coordinates} which are depicted as black dotted lines. Their upper branches take a much higher excursion as indicated from the order of $\rho_{\epsilon}$. Notice that as $\epsilon$ becomes smaller, this upper branch tends to straighten horizontally even more.}
        \label{fig:limitcycle_comp}
\end{figure}
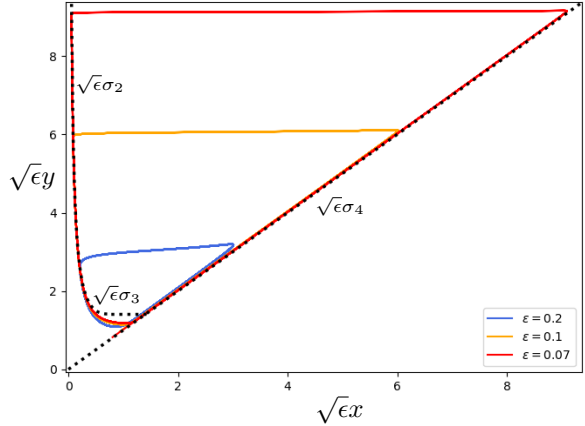

\subsection{Analysis near infinity}
The set $\left\{ (x,y)\in\mathbb{R}_+^2 : y\geq x  \right\}$ is absorbing and positively invariant under the dynamics of the slow-fast system \eqref{eq:slowDynamics}-\eqref{eq:fastDynamics}. 
Due to the geometry of the state space, and our goal to describe a limit cycle, let us consider the system in polar-like coordinates
\begin{equation}
\label{eq: ChangeCoords_First}
    x=\frac{\cos\theta}{r}, \qquad y=\frac{\sin\theta}{r},
\end{equation}
so that \eqref{eq:fastDynamics} becomes the singularly perturbed system in nonstandard form
\begin{equation}
\label{eq:slowfast_polar}
\begin{array}{lcl}
\dot{\theta} &=& -a\sin\theta(\sin\theta-\cos\theta) \\
& & +\epsilon\left[ \frac{\sin\theta\cos\theta(\sin\theta-\cos\theta)^2}{r^2}-\cos\theta(\sin\theta-\cos\theta) +ar\cos\theta \right],\\

\dot{r}&=&-ar\cos\theta(\sin\theta-\cos\theta) \\
& & +
\epsilon \left[ \frac{\cos^2\theta(\sin\theta-\cos\theta)^2}{r}+r\sin\theta(\sin\theta-\cos\theta)-ar^2\sin\theta \right],
\end{array}
\end{equation}
where the corresponding phase space is given by $\{ (r,\theta) : r\geq 0, \theta\in[\pi/4,\pi/2] \}$ and $r=0$ corresponds to the \textit{line at infinity} for system \eqref{eq:fastDynamics}.

In order to extend the state space by $\{r=0\}$,
we perform the time reparametrisation 
\begin{equation}
\label{eq:time_change}
    t_1(t)=\int_{0}^t{\frac{1}{r^2(s)}\;ds},
\end{equation}
yielding thus the ODE system
\begin{equation}
    \label{eq:slowDynamicsCompact}
    \begin{array}{lll}
    \dot{\theta}&=&-ar^2\sin\theta(\sin\theta-\cos\theta) \\
    & & + \epsilon \left [
    \sin\theta\cos\theta(\sin\theta-\cos\theta)^2- r^2\cos\theta(\sin\theta-\cos\theta)+ar^3\cos\theta\right]\\
    \dot{r} &=& -ar^3\cos\theta(\sin\theta-\cos\theta) \\
    & & +\epsilon\left[ r\cos^2\theta(\sin\theta-\cos\theta)^2+ r^3\sin\theta(\sin\theta-\cos\theta)- ar^4\sin\theta \right],
    \end{array}
\end{equation}
where, by abuse of notation, we will take $\dot{\;} \equiv d/dt_1$. 

Firstly, we analyse the layer problem of \eqref{eq:slowDynamicsCompact} by taking $\epsilon=0$, for which the system reads as
\begin{equation}
    \label{eq:layerproblem}
    \begin{array}{rcl}
        \dot{\theta} &=&-ar^2\sin\theta(\sin\theta-\cos\theta), \\
        \dot{r} &=& -ar^3\cos\theta(\sin\theta-\cos\theta).
    \end{array}
\end{equation}
A parametrisation of the orbits of \eqref{eq:layerproblem} can be obtained from the equation
\[
        \frac{dr}{d\theta}=r\cot\theta,
\]
whose solutions are given by $r(\theta)=r_0\sin\theta$ for $\theta\in[\pi/4,\pi/2]$, where $r(\pi/2)=r_0$. 

It follows directly from \eqref{eq:layerproblem} that the lines 
\[
    \{\theta=\pi/4\}, \quad \{r=0\}
\]
consist entirely of equilibria. Let $J(\theta,r)$ be the Jacobian matrix of the vector field on the right hand side of equation~\eqref{eq:layerproblem}. 
For the respective equilibrium points, we obtain
\[
J(\pi/4,r)=\left[
\begin{array}{cc}
     -ar^2 & 0 \\
     -ar^3& 0 
\end{array}\right], \qquad J(\theta,0)=\left[
\begin{array}{cc}
     0 & 0 \\
     0 & 0 
\end{array}\right].
\]
Since $\{(\pi/4,r)\}$ is partially hyperbolic away from $r=0$, standard singular perturbation theory provides the necessary tools to analyse the system along such a line. However, this is not the case for the line $\{r=0\}$ which consists of fully nonhyperbolic equilibria. In order to extend the analysis further, and to comprehend the reduced problem as well, we perform an $\epsilon$-dependent rescaling of the variable $r$ which desingularises the nonhyperbolic line.

\begin{figure}[ht]
     \centering
      \begin{overpic}[width=.5\textwidth]{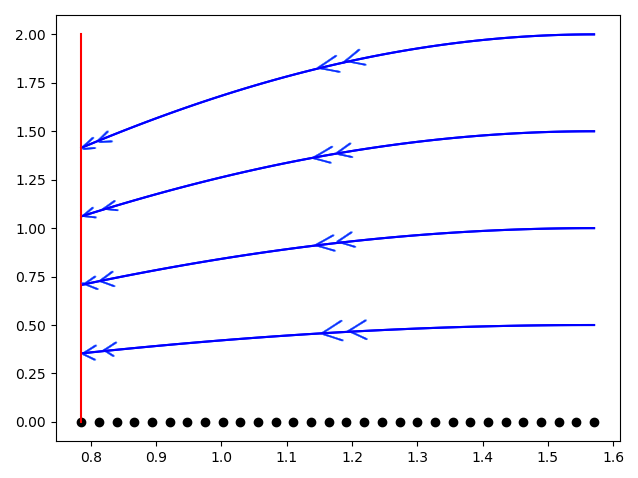}
        \put(53,-2){ $\theta$}
        \put(-3,38){ $r$}
        \end{overpic}
       \hfill
       \caption[Stable manifolds - critical case]{Dynamics of the layer problem \eqref{eq:layerproblem}. Each fast fibre, on which the dynamics of the layer problem converge to different points $(r,\pi/4)$ with $r>0$, are depicted in blue. The arc 
       $\{r=0\}$ consists entirely of nonhyperbolic fixed points.}
        \label{fig:StableManifolds}
\end{figure}

\section{Desingularisation of the system}
\label{SECTION:blowup}
We introduce the singular rescaling
\begin{equation}
    \label{eq:rescale}
        \begin{array}{rcl}
             r&=&\sqrt{\epsilon}\;\bar{r}, \qquad 
             t_2=\epsilon t_1,
             \end{array}
\end{equation}
so that, by writing $r$ instead of $\bar r$ for ease of notation,
\eqref{eq:slowDynamicsCompact} reads as
\begin{equation}
    \label{eq:MainBrusselator}
    \begin{array}{lll}
    \dot{\theta}&=&-ar^2\sin\theta(\sin\theta-\cos\theta) + \sin\theta\cos\theta(\sin\theta-\cos\theta)^2+
    \epsilon g_1(\theta,r,\epsilon)\\
    \dot{r}&=&-ar^3\cos\theta(\sin\theta-\cos\theta)+r\cos^2\theta(\sin\theta-\cos\theta)^2+ \epsilon g_2(\theta,r,\epsilon),
    \end{array}
\end{equation}
where, again, we transfer the notation $\dot{\;}:=d/dt_2$ and
\[
\begin{array}{rcl}
    g_1(\theta,r,\epsilon)&=&
    -r^2\cos\theta(\sin\theta-\cos\theta) +\sqrt{\epsilon}\;ar^3\cos\theta, \\
    g_2(\theta,r,\epsilon) &=& r^3\sin\theta(\sin\theta-\cos\theta) -\sqrt{\epsilon}\;ar^4\sin\theta.
\end{array}
\]
Moreover, since $\epsilon\ll1$, \eqref{eq:MainBrusselator} is again a singularly perturbed system in nonstandard form. As opposed to \eqref{eq:slowDynamics}, the line $\{r=0\}$ corresponds now to a heteroclinic orbit connecting $(\pi/4,0)$ to $(\pi/2,0)$.

\subsection{The layer problem}
\label{SUBSEC: layer}
The layer problem of \eqref{eq:MainBrusselator}, i.e. when $\epsilon=0$, is given by the ODEs
\begin{equation}
    \label{eq:critical}
    \begin{array}{rcl}
        \dot{\theta}&=&-\sin\theta (\sin\theta-\cos\theta)\cdot p(\theta,r),
        \\
        \dot{r} &=& -r\cos\theta (\sin\theta-\cos\theta)\cdot p(\theta,r),
    \end{array}
\end{equation}
where
\begin{equation}
\label{eq:pfunc}
p(\theta,r):=ar^2-\cos\theta(\sin\theta-\cos\theta).
\end{equation}
The two branches of the critical manifold are thus given by the equilibria of \eqref{eq:critical}, namely
\begin{equation}
\label{eq:criticalmanifolds}
    \mathcal{S}^1_0:=\{ (\pi/4,r) : r\geq 0 \}, \qquad \mathcal{S}^2_0:=\left\{  (\theta,\phi_0(\theta)) : \phi_0(\theta)=\sqrt{\frac{\cos\theta (\sin\theta-\cos\theta)}{a}} \right\}.
\end{equation}
By direct calculation we can provide further information about the hyperbolicity of $\mathcal{S}^1_0$ and $\mathcal{S}^2_0$.
\begin{proposition}
    \label{PROP:hyperbolicityChart2}
    Let $h(\theta,r)$ be the vector field which defines \eqref{eq:critical}. Then, the Jacobian matrix $Dh$ satisfies that:
    \begin{enumerate}
        \item in $\mathcal{S}^1_0$, it has a simple zero eigenvalue and a negative simple eigenvalue for all $r>0$. For $r=0$, zero is a double eigenvalue.
        \item in $\mathcal{S}^2_0$, for all $\theta\not\in \{\pi/4,\arctan(2)\}$ it has a simple zero and another simple eigenvalue, which is positive when $\theta\in(\pi/4,\arctan(2))$ and negative when $\theta\in(\arctan(2),\pi/2)$. Otherwise, it has a double zero eigenvalue.
    \end{enumerate}
\end{proposition}
\begin{proof} 
 By direct calculation, $Dh$ on $\mathcal{S}^1_0$ reads
\[
   Dh(\pi/4,r)= \left[\begin{array}{cc}
    -ar^2 & 0 \\
    -ar^3 & 0 \\
    \end{array}\right],
\]
so that the nontrivial simple eigenvalue is given by $-ar^2$, which is negative whenever $r\neq 0$.

On the other hand, on $\mathcal{S}^2_0$ the Jacobian matrix, after evaluating it at $r=\phi_0(\theta)$, is given by
\[
    Dh(\theta,\phi_0)= \left[\begin{array}{cc}
        -\sin\theta(\sin\theta-\cos\theta) \cdot\partial_\theta p&

        -\sin\theta(\sin\theta-\cos\theta)\cdot\partial_r p\\

    -\phi_0\cdot \cos\theta(\sin\theta-\cos\theta)\cdot\partial_\theta p & -\phi_0\cdot\cos\theta(\sin\theta-\cos\theta)\cdot \partial_r p
        
    \end{array}\right].
\]
Observe that, since $\det Dh(\theta,\phi_0)=0$ for all $\theta\in[\pi/4,\pi/2]$, one of the eigenvalues is zero. The potentially nontrivial eigenvalue is given by its trace, namely
\begin{align*}
tr\; Dh(\theta,\phi_0)= & -\sin\theta(\sin\theta-\cos\theta) \left[
\sin\theta(\sin\theta-\cos\theta) -\cos\theta(\sin\theta+\cos\theta)\right]\\
&-2\cos^2\theta(\sin\theta-\cos\theta)^2.    
\end{align*}

Since $tr\; Dh(\theta,\phi_0)$ has $(\sin\theta-\cos\theta)$ as a factor, it becomes zero at $\theta=\pi/4$. Otherwise,
$tr\; Dh(\theta,\phi_0)<0$ if and only if
\begin{equation}
    \label{eq:step}
    -\sin^3\theta+2\sin^2\theta\cos\theta-\sin\theta\cos^2\theta+2\cos^3\theta<0.
\end{equation}
By substituting $\sin^2\theta=1-\cos^2\theta$ and $\sin^3\theta=\sin\theta-\sin\theta\cos^2\theta$, we get that \eqref{eq:step} is equivalent to
\[
    \tan\theta>2,
\]
and the result follows.
\end{proof}

\begin{remark}
\label{RMK:foldpoint}
The geometry of the fast fibres of \eqref{eq:critical} can be obtained easily by taking $\dot{r}/\dot{\theta}$, so that
\begin{equation}
\label{eq:orbitODE}
    \frac{dr}{d\theta}=r\cot\theta.
\end{equation}
While the solutions are given as $r_{\rho}(\theta)=\rho\sin\theta$ as for \eqref{eq:layerproblem}, where $\rho\geq 0$, the dynamics are of different nature. Indeed, the graphs of $r_{\rho}$ and the critical manifold $\mathcal{S}_0^2$ may intersect depending on $\rho$. The intersections can be calculated by solving the equation
\[
    \rho\sin\theta=\sqrt{\frac{\cos\theta(\sin\theta-\cos\theta)}{a}},
\]
which is equivalent to the quadratic equation
\[
    a\rho^2\tan^2\theta-\tan\theta+1=0,
\]
whose solutions are given by
\[
    \tan{\theta^{\pm}}=\frac{1\pm \sqrt{1-4a\rho^2}}{2a\rho^2}.
\]
Note that the solutions exist if and only if $\rho\leq \rho^*:=1/(2\sqrt{a})$, and that for each point $z=(\theta,\phi_0(\theta))\in \mathcal{S}^2_0$ there is a unique $\rho$ such that the graph of $r_{\rho}(\theta)$ intersects $\mathcal{S}_0^2$ precisely at $z$. The solutions of \eqref{eq:orbitODE} constitute the local stable or unstable manifolds $W^s_{loc}(z), W^u_{loc}(z)$ accordingly, see Figure~\ref{fig:chart2}. A tangency occurs at the point $F=(\theta^*,r^*)$ with $\theta^*=\arctan(2)$ and $r^*=\phi_0(\theta^*)\equiv 1/\sqrt{5a}$, where  $\mathcal{S}^2_0$ loses hyperbolicity. As a summary, the graphs of $r_{\rho}$ constitute the stable/unstable foliation $\mathcal{F}_0$ of any compact submanifold of the respective branch $\mathcal{S}^1_0$ or $\mathcal{S}^2_0$ which does not contain $(\pi/4,0)$ or $(\theta^*,R^*)$. In the following, we will denote the fast fibration by 
\begin{align*}
    \mathcal{F}_0(\rho)= &\{(\theta,r_{\rho}(\theta)): \theta\in[\pi/4,\pi/2]\}, \\
    \mathcal{F}_0\equiv &\bigcup_{\rho\geq 0}\mathcal{F}_0(\rho).
\end{align*}
\end{remark}

\subsection{The reduced problem}
\label{SUBSEC:reduced}
In accordance with the setup in \cite{Wechselberger2020}, system \eqref{eq:MainBrusselator} can be factorised as 
    \begin{equation}
    \label{eq:factorised}
        \left(\begin{array}{c}
             \dot{\theta} \\
             \dot{r}
        \end{array} \right)=N(\theta,r) f(\theta,r) + \epsilon\; G(\theta,r;\epsilon),
    \end{equation}
where
\[
    N(\theta,r)=\left( \begin{array}{c}
         -\sin\theta  \\
          -r\cos\theta
    \end{array} \right), \qquad 
    f(\theta,r)=(\sin\theta-\cos\theta)\cdot p(\theta,r),
\]
\[
    G(\theta,r;\epsilon)= \left( \begin{array}{c}
         -r^2\cos\theta(\sin\theta-\cos\theta) +\sqrt{\epsilon}\;ar^3\cos\theta \\
          r^3\sin\theta(\sin\theta-\cos\theta) -\sqrt{\epsilon}\; ar^4\sin\theta
    \end{array} \right).
\]
Since $N\neq 0$ for all $\theta\in[\pi/4,\pi/2]$, the critical manifolds $\mathcal{S}^1_0$ and $\mathcal{S}^2_0$ coincide with the zero-level set of $f$. By considering the time rescaling $\tau_2=\epsilon\; t_2$, we obtain the equivalent system
\begin{equation}
\label{eq:reduced}
     \left(\begin{array}{c}
             \theta' \\
             r'
        \end{array} \right)=\frac{1}{\epsilon}N(\theta,r) f(\theta,r) + G(\theta,r;\epsilon),
\end{equation}
where $\cdot'=d/d\tau_2$. We analyse the reduced problem on $\mathcal{S}_0^2$ in the next proposition, recalling the definition of its parametrisation map $\phi_0$ from~\eqref{eq:criticalmanifolds}. 
\begin{proposition}
    \label{PROP:reduced_problem}
    The dynamics of the reduced problem on $\mathcal{S}^2_0 \setminus \{F\}$ is governed by the ODE
    \begin{equation}
        \label{eq:reduced_S2}  \theta'= \frac{2\phi_0^2(\theta)\cos\theta(\sin\theta-\cos\theta)^2}{2\cos\theta-\sin\theta}, \ \theta(0) \neq \arctan 2.
    \end{equation}
\end{proposition}
\begin{proof}
Due to the factorisation \eqref{eq:factorised}, by means of Lemma 3.4 in \cite{Wechselberger2020}, the reduced problem is given by
\begin{equation}
\label{eq:reduced_Wechselberger}
        \left(\begin{array}{c}
             \theta' \\
             r'
        \end{array} \right)=\left(Id-\frac{1}{\langle\nabla f,N\rangle}N \nabla f\right)G(\theta,r;0),
\end{equation}
where $Id$ denotes the identity matrix, and the equation should be restricted to $\mathcal{S}_0^i$. In particular, when evaluating in $(\theta,\phi_0(\theta))$, we obtain
\[
    \nabla f(\theta,\phi_0)=(\sin\theta-\cos\theta) \cdot \left(\sin\theta(\sin\theta-\cos\theta)-\cos\theta(\sin\theta+\cos\theta), \; 2a\phi_0
    \right),
\]
and can calculate
\[
    \left\langle\nabla f(\theta,\phi_0), N(\theta,\phi_0)\right\rangle=(\sin\theta-\cos\theta)\cdot(2\cos\theta-\sin\theta).
\]
Furthermore, we compute
\[
    N\nabla f=(\sin\theta-\cos\theta)\left[ 
    \begin{array}{cc}
    -\sin^2\theta(\sin\theta-\cos\theta)+\sin\theta\cos\theta(\sin\theta+\cos\theta) & -2a\phi_0\sin\theta \\
    -\phi_0\sin\theta\cos\theta(\sin\theta-\cos\theta)+\phi_0\cos^2\theta(\sin\theta+\cos\theta) & -2a\phi^2\cos\theta
    \end{array}\right].
\]
Considering the equation for $\theta$ in~\eqref{eq:reduced_Wechselberger} then yields
\[
    \dot{\theta}=\frac{\phi_0^2(\sin\theta-\cos\theta)}{2\cos\theta-\sin\theta}\left( 
    -\cos\theta(2\cos\theta-\sin\theta)+\sin^3\theta\cos\theta+\sin\theta\cos^3\theta
    \right).
\]
By substituting $\sin^3\theta=\sin\theta-\sin\theta\cos^2\theta$, we obtain equation \eqref{eq:reduced_S2}.
\end{proof}

    It is now easy to observe from equation \eqref{eq:reduced_S2} that the reduced flow on $\mathcal{S}_0^2$ leads to the fold point $F$ from any initial point on $\mathcal{S}_0^2 \setminus \{F, (\pi/4, 0), (\pi/2, 0)\}$.

\begin{remark}
    From \eqref{eq:reduced_Wechselberger} and since $G(\pi/4,r;0)\equiv 0$, the reduced flow on $\mathcal{S}^1_0$ vanishes. In other words, the first order approximation provides no further information on the dynamics near $\mathcal{S}_0^1$ and a time rescaling of higher order in $\epsilon$ is needed. Indeed, from \eqref{eq:MainBrusselator}, substituting $\theta=\pi/4$ and projecting the resulting vector field to $T\mathcal{S}_0^1$ yields the ODE
    \[
        \dot{r}=-\frac{\epsilon^{3/2}ar^4}{\sqrt{2}}.
    \]
    It is readily evident that the superslow flow along $\mathcal{S}^1_0$ points towards the nonhyperbolic point $P_0$, and that the time scale spent near $\mathcal{S}^1_0$ is $\Theta\left( \epsilon^{-3/2} \right)$.
\end{remark}

\begin{figure}[ht]
     \centering
      \begin{overpic}[width=.55\linewidth]{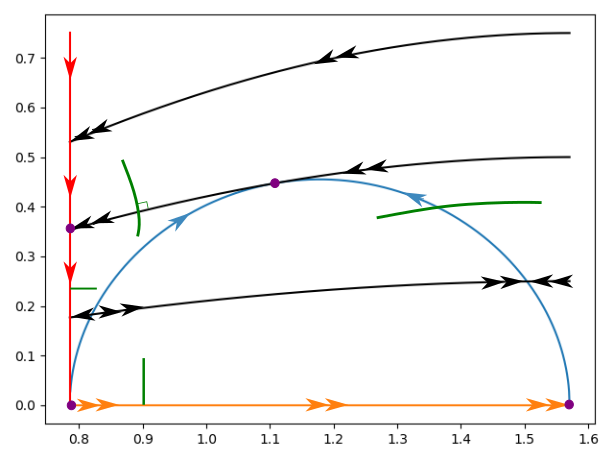}
      \put(53,-2){\scriptsize $\theta$}
    \put(25,13){\scriptsize $\Sigma_1$}
    \put(-3,38){\scriptsize $r$}
    \put(85,39){\scriptsize $\Sigma_2$}
    \put(22,45){\scriptsize $\Sigma_3$}
    \put(13,30){\scriptsize $\Sigma_4$}
    \put(-3,38){\scriptsize $r$}
    \put(12,11){\scriptsize $P_0$}
    \put(93,11){\scriptsize $P_1$}
    \put(43,48){\scriptsize $F$}
    \put(12,40){\scriptsize $Q$}
    \put(12,67){\scriptsize $\mathcal{S}_0^1$}
    \put(85,22){\scriptsize $\mathcal{S}_0^2$}
      \end{overpic}
       \hfill
       \caption[Dynamics in second chart]{Dynamics of the layer problem \eqref{eq:critical} and of the reduced problem \eqref{eq:reduced}. The branches $\mathcal{S}_0^1$ and $\mathcal{S}^2_0$ are depicted in red and blue, respectively. The fast fibres $\mathcal{F}_0(\rho)$ are portrayed in black, on which the direction of the fast flow is indicated with double arrows. On the contrary, the direction of the reduced flow is given on the critical manifolds with single arrows. Transverse sections are depicted in green, on which transition maps are defined in order to construct a Poincaré map, as done in Section~\ref{SECTION:poincareMap}.  The singular cycle consists of the curves $\hat{\sigma}_1$ connecting $P_0$ to $P_1$ (in orange), $\hat{\sigma}_2$ branch of $\mathcal{S}_0^2$ connecting $P_1$ to $F$, $\hat{\sigma}_3$ the fast fibre connecting $F$ to $Q$, and $\hat{\sigma}_4$ the branch of $\mathcal{S}_0^1$ connecting $Q$ to $P_0$.}
        \label{fig:chart2}
\end{figure}
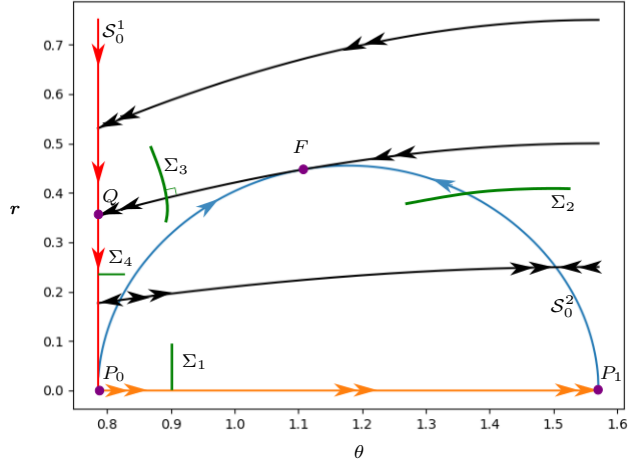

\subsection{The singular cycle}
From the analysis in Subsections \ref{SUBSEC: layer} and \ref{SUBSEC:reduced} for the layer problem and the reduced problem of \eqref{eq:MainBrusselator} respectively, we construct a singular cycle which approximates the limit cycle for $\epsilon$ small enough. Consider the points  $P_0=(\pi/4,0)$, $P_1=(\pi/2,0)$, the fold point $F=(\theta^*,r^*)$, and $Q=\left(\frac{\pi}{4},\frac{1}{2\sqrt{2a}}\right)$. The singular cycle $\hat{\sigma}$ consists of the following curves:
\begin{itemize}
    \item[$\hat{\sigma}_1$:] the heteroclinic orbit connecting $P_0$ and $P_1$,
    \item[$\hat{\sigma}_2$:] the submanifold in $\mathcal{S}_0^2$, connecting $P_1$ to $F$,
    \item[$\hat{\sigma}_3$:] the fast fibre $\mathcal{F}_0(\rho_*)$ with $\rho_*=1/(2\sqrt{a})$, from its tangency point $F$ with $\mathcal{S}_0^2$, to the drop point $Q$ in $\mathcal{S}_0^1$,
    \item[$\hat{\sigma}_4$:] the submanifold in $\mathcal{S}_0^1$, connecting $Q$ to $P_0$.
\end{itemize}
A complete portrait of the singular cycle is presented in Figure~\ref{fig:chart2}.

\section{Construction of a Poincaré map}
\label{SECTION:poincareMap}
In the following we consider a vertical section
\begin{equation}
    \label{eq:transverseFirst}
        \Sigma_1:=\left\{ (\pi/4+\alpha_1, r) : r\in[0,\beta_1] \right\},
\end{equation}
for sufficiently small constants $\alpha_1,\beta_1$. We devote this section to showing that a first return map $\Pi^{\epsilon}:\Sigma_1\rightarrow\Sigma_1$ can be well defined as indicated in the following theorem. 
\begin{theorem}
    \label{THM:PoincareMap}
        There exist $\Sigma_1$ as in \eqref{eq:transverseFirst} and $\epsilon_0$ sufficiently small such that for every $\epsilon\in[0,\epsilon_0]$, the Poincaré map $\Pi^{\epsilon}: \Sigma_1\rightarrow\Sigma_1$  is well defined along solutions of~\eqref{eq:MainBrusselator}. Moreover, for each $\epsilon\in(0,\epsilon_0]$
        \begin{itemize}
            \item[(i)] system \eqref{eq:MainBrusselator} admits a unique attracting limit cycle $\gamma_\epsilon$ which converges in Hausdorff distance\footnote{Recall that the Hausdorff semidistance between sets is defined as
            \[ \TU{dist}_H(A,B):=\sup_{a\in A}\TU{d}(a,B)=\sup_{a\in A}\inf_{b\in B}\TU{d}(a,b),\]
            and thus the Hausdorff distance $\TU{d}_H$ is defined as $\TU{d}_H(A,B)=\max\{ \TU{dist}_H(A,B), \TU{dist}_H(B,A)\}$.} to the singular cycle $\hat{\sigma}=\hat{\sigma}_1\cup\hat{\sigma}_2\cup\hat{\sigma}_3\cup\hat{\sigma}_4$ as $\epsilon\rightarrow0$,
            \item[(ii)] the map $\Pi^{\epsilon}$ is such that for each $z\in\Sigma_1$, 
            \[
                \Pi^\epsilon(z)=\left( \pi/4+\alpha_1,\Theta\left( \epsilon^{3/2}\right) \right),
            \]
            which is a contraction of order $\mathcal{O}\left( e^{-c/\epsilon^3}\right)$, 
            \item[(iii)]it admits a unique fixed point $z_\epsilon\in\Sigma_1$ such that 
            \[z_{\epsilon}=z_0+\Theta \left( \epsilon^{3/2}\right),\] where $z_0=(\pi/4+\alpha_1,0)$, and
            \item[(iv)] the dynamics on the limit cycle $\gamma_\epsilon$ admits a time scale separation where
            \begin{itemize}
                \item the time time scale near $\hat{\sigma}_1$ is of order $\Theta(1)$,
                \item the time time scale near $\hat{\sigma}_2$ is of order $\Theta\left(\epsilon^{-1}\right)$,
                \item the time time scale near $\hat{\sigma}_3$ is of order $\Theta(1)$, and
                \item the time time scale near $\hat{\sigma}_4$ is of order $\Theta\left(\epsilon^{-3/2}\right)$.
            \end{itemize}
        \end{itemize}
\end{theorem}
Theorem~\ref{THM:PoincareMap} implies that the self-sustained oscillations of the Brusselator are of a relaxation nature by switching between fast and slow regimes. More specifically, the limit cycle splits into fast regimes near $\hat{\sigma}_1$ and $\hat{\sigma}_3$, a slow regime near $\hat{\sigma}_2$, and a superslow one near $\hat{\sigma}_4$.

Similarly to \cite{Szmolyan09}, in order to prove Theorem~\ref{THM:PoincareMap}, we define appropriate transverse sections $\Sigma_i$ across each $\hat{\sigma}_i$, $i=1,2,3,4$, and transition maps $\Pi_{ij}^{\epsilon}\equiv\Pi_{ij}:\Sigma_i\rightarrow{\Sigma_j}$, where $\Pi_{ij}(z)=\tilde{z}$ if and only if $\tilde{z}$ is the point where the orbit starting in $z\in\Sigma_i$ intersects $\Sigma_j$ for the first time. Therefore, we dedicate this section to showing that
\[
\begin{array}{c}
    \Pi^{\epsilon}:\Sigma_1\rightarrow\Sigma_1\\
    \Pi^{\epsilon}=\Pi_{41}\circ\Pi_{34}\circ
    \Pi_{23}\circ\Pi_{12}
\end{array}
\]
is a well defined return map, which admits a unique fixed point, as described in Theorem~\ref{THM:PoincareMap}.

\subsection{The transition map \texorpdfstring{$\Pi_{12}$}{P12} -- contraction to the slow manifold \texorpdfstring{$\mathcal{S}^2_{\epsilon}$}{S2e}}
Recall from Remark~\ref{PROP:hyperbolicityChart2} that every $(\theta,\phi_0(\theta))\in\mathcal{S}^2_0$ for $\theta>\theta^*$ is partially hyperbolic, and its stable manifold is given by $r_{\rho}(\theta)=\rho\sin\theta$ for a unique $\rho\in[0,\rho^*)$. Due to the smoothness of system \eqref{eq:MainBrusselator}, standard Fenichel theory \cite[Theorem 2.4.2]{Kuehn15} guarantees that for each $\tilde{\theta}\in(\theta^*,\pi/2)$
there is a slow manifold $\mathcal{S}^2_{\epsilon}$ for each $\epsilon$ sufficiently small which is locally the graph of a function $\phi_{\epsilon}:\left[\tilde{\theta},\pi/2\right]\rightarrow \mathbb{R}$, admitting a stable foliation $\mathcal{F}_{\epsilon}$ along which there is exponential contraction towards $\mathcal{S}^2_\epsilon$. The expression of $\phi_{\epsilon}$ up to first order approximation is given in the next proposition.
\begin{proposition}
    \label{PROP:expansionSlowManifold}
        Any slow manifold $\mathcal{S}^2_{\epsilon}$ is locally the graph of a function \[\phi_{\epsilon}=\phi_0+\epsilon\phi_1+o(\epsilon),\]
        where $\phi_0$ is as in \eqref{eq:criticalmanifolds}, and
        \begin{equation}
            \label{eq:expansiontermsCritManifold}
            \phi_1(\theta)= -\frac{\phi_0(\theta)\cos\theta}{2a(2\cos\theta-\sin\theta)}.
        \end{equation}
    In particular, $\phi_{\epsilon}=\phi_0 +\Theta\left( \epsilon\right)$ for $\theta$ bounded away from $\pi/2$ and $\theta^*$. 
\end{proposition}
\begin{proof}See Appendix~\ref{SEC:expansions}. \end{proof}
\begin{remark}
    \label{RMK:centremanifold_unique}
    Any slow manifold $\mathcal{S}^2_{\epsilon}$ is in fact a \textit{centre manifold} $W^c_{\epsilon}(P_1)$ based on $P_1=(\pi/2,0)$, since $P_1$ is partially hyperbolic for every $\epsilon>0$. Indeed, let $H(\theta,r)$ be the vector field defining \eqref{eq:MainBrusselator}. Then, the Jacobian matrix $DH$ at $P_1$ is
    \[
        DH(\pi/2,0)=\left[ 
        \begin{array}{cc}
            -1 & 0 \\
            0 & 0
        \end{array}
        \right].
    \]
    The existence of $W^c_\epsilon$ is guaranteed by the \textit{centre manifold theorem}, see for instance \cite[Theorem 5.1]{kuznetsov04} or \cite[Theorem 5.1]{Turaev98}.

    While centre manifolds are in general nonunique, in this case they are uniquely defined for $r\geq 0$, which is contained in our region of interest.
\end{remark}

\begin{remark}
    \label{RMK:timescale1}
Proposition~\ref{PROP:expansionSlowManifold} implies that the time scale near $\mathcal{S}^2_{\epsilon}$ is of order $\Theta(\epsilon^{-1})$, see Appendix~\ref{SEC:expansions}.
\end{remark}

For defining the transversal section $\Sigma_2$, we make use of the fast fibres $\mathcal{F}_0(\rho)$ setting
\begin{equation}
    \label{eq:transverseSecond}
        \Sigma_2:=\left\{ (\theta, \beta_2\sin\theta)) : \theta\in\left[\theta^*+\frac{\alpha_2}{2},\theta^*+\frac{3\alpha_2}{2}\right]
        \right\},
\end{equation}
where
\[
    \beta_2=\frac{\phi_0(\theta^*+\alpha_2)}{\sin(\theta^*+\alpha_2)}
\]
and $\alpha_2>0$ is sufficiently small, see Figure~\ref{fig:fold}. In essence, $\Sigma_2$ is a small segment of $\mathcal{F}_0(\beta_2)$ around $\theta^*+\alpha_2$, such that $\phi_0(\theta^*+\alpha_2)=r_{\beta_2}(\theta^*+\alpha_2)$. Note that $\alpha_2$ being small implies $\beta_2\approx\rho_*$.

Since the stable foliation $\mathcal{F}_0$ is transversal to $\Sigma_1$, this remains true for $\mathcal{F}_\epsilon$ for all $\epsilon$ sufficiently small, and thus the orbit starting in a point $(\pi/4+\alpha_1,r)\in\Sigma_1$ is attracted exponentially fast to $\mathcal{S}^2_{\epsilon}$. The transition map $\Pi_{12}$ can be extended to $z_0=(\pi/4+\alpha_1,0)$ by mapping this point to $\mathcal{S}^2_{\epsilon}\cap \Sigma_2=\left(\theta^*+\alpha_2, \phi_0(\theta^*+\alpha_2) \right)+\Theta(\epsilon)$. As a summary we have the following lemma:
\begin{lemma}
\label{LEMMA:Transition1}
    For some $\alpha_1,\alpha_2, \beta_1$ small enough the transition map $\Pi_{12}:\Sigma_1\rightarrow\Sigma_2$ is well defined for every $\epsilon$ sufficiently small, where 
    \[\Pi_{12}(z_1)=\left(\theta^*+\alpha_2, \phi_0(\theta^*+\alpha_2) \right)+\Theta(\epsilon),\]
    for all $z_1\in\Sigma_1$.
    Moreover, the transition map $\Pi_{12}$ is a contraction with a contraction rate of order $\mathcal{O}(e^{-c_1/\epsilon})$ for some $c_1>0$. 
\end{lemma}

\subsection{The transition map \texorpdfstring{$\Pi_{23}$}{P23} -- passage through the fold}
\label{SEC:PI_23}
Recall from Remark~\ref{RMK:foldpoint} that normal hyperbolicity of the critical manifold $\mathcal{S}^2_0$ is lost at the point $F=(\theta^*,r^*)$, and the passage through this fold point is to be analysed. 

Let us consider $\alpha_3>0$ sufficiently small, and the curvilinear segment $\Sigma_3$,
which 
\begin{enumerate}
\item passes through $(\theta^*-\alpha_3,r_{\rho_*}(\theta^*-\alpha_3))\in\mathcal{F}_0(\rho_*)$,
\item is perpendicular to each fibre $\mathcal{F}(\rho)$ with $\rho\approx\rho_*$, and
\item lies above and away from $\mathcal{S}_0^2$.
\end{enumerate}
  For a clear depiction of $\Sigma_3$, see Figure~\ref{fig:fold}. A direct application of classical transitions through folds is possible due to Proposition~\ref{PROP:reduced_problem}, as given in the next statement.

\begin{lemma}
\label{LEMMA:transition2}
    For each $\Sigma_3$ as above, there exists $\Sigma_2$ as in (\ref{eq:transverseSecond}) such that the transition map $\Pi_{23}:\Sigma_2\rightarrow\Sigma_3$ is well defined for every $\epsilon$ sufficiently small. The slow manifold $\mathcal{S}_\epsilon^1$ passes through a point $(\theta^*-\alpha_3,r_{\rho_*}(\theta^*-\alpha_3))+\Theta(\epsilon^{2/3})$.
    Moreover, the map $\Pi_{23}$ is a contraction with a contraction rate of order $\mathcal{O}(e^{-c_2/\epsilon})$, for some $c_2>0$, and thus for every $z_2\in\Sigma_2$,
    \[
        \Pi_{23}(z_2)=(\theta^*-\alpha_3, r_{\rho_*}(\theta^*-\alpha_3))+\Theta \left( \epsilon^{2/3} \right).
    \]
\end{lemma}
\begin{proof}
    Recall that the fold point $F=(\theta^*,r^*)$ is given by the tangency between the fibre $\mathcal{F}_0(\rho_*)$ and $\mathcal{S}^2_0$, which are given by the graphs of $r_{\rho_*}(\theta)=\rho_*\sin\theta$ and $\phi_0(\theta)$, respectively. This tangency is of order 1, that is
    \[
        r^*=r_{\rho_*}(\theta^*)=\phi_0(\theta^*), \quad r_{\rho_*}'(\theta^*)=\phi'_0(\theta^*), \quad r_{\rho_*}''(\theta^*)\neq\phi''_0(\theta^*).
    \]
    By Proposition~\ref{PROP:reduced_problem}, the reduced flow points towards $F$, and the result follows from
    \cite[Lemma 4.3 and Theorem 4.2]{Wechselberger2020}.
\end{proof}

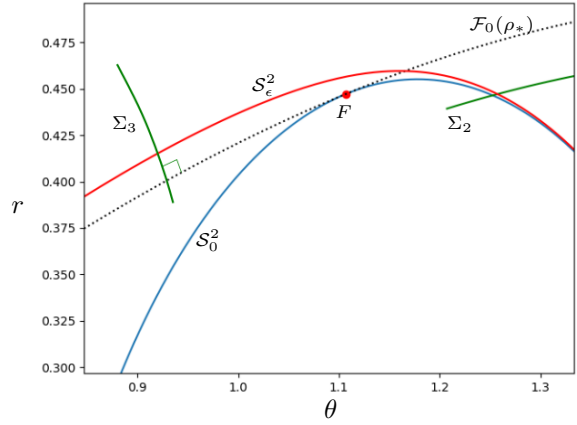
\begin{figure}[ht]
     \centering
      \begin{overpic}[width=.5\linewidth]{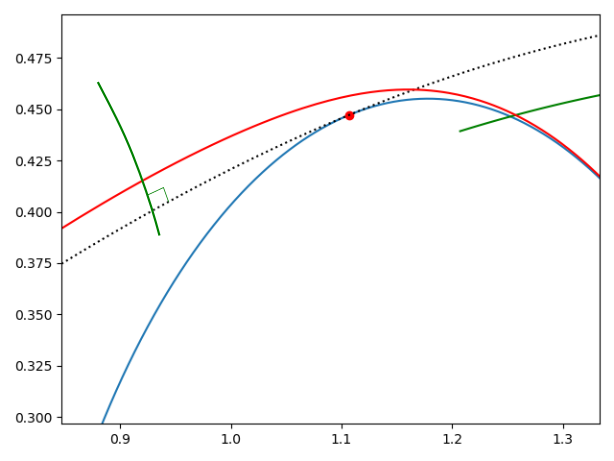}
    \put(53,-2){$\theta$}
    \put(-3,35){$r$}
    \put(15,50){\scriptsize $\Sigma_3$}
    \put(40,57){\scriptsize $\mathcal{S}_{\epsilon}^2$}
    \put(55,52){\scriptsize $F$}
    \put(75,50){\scriptsize $\Sigma_2$}
    \put(79,68){\scriptsize $\mathcal{F}_0(\rho_*)$}
    \put(30,29){\scriptsize $\mathcal{S}_0^2$}
      \end{overpic}
       \hfill
       \caption[Dynamics near the fold]{The phase portrait near the fold point $F$. The segments in green are the sections $\Sigma_2$ (on the right) and $\Sigma_3$ (on the left). In solid blue the graph of $\phi_0$ is depicted. In
       dotted lines the fast fibre $\mathcal{F}_0(\rho_*)$ is presented as a reference. In red solid line the extension of the slow manifold $\mathcal{S}^2_{\epsilon}$ is portrayed. The deviation from the fast fibre $\mathcal{F}_0({\rho_*})$ is of order $\Theta\left( \epsilon^{2/3}\right)$.}
        \label{fig:fold}
\end{figure}

\subsection{The transition map \texorpdfstring{$\Pi_{34}$}{P34}}

Let us consider a transverse section
\begin{equation}
    \label{eq:Section4}
        \Sigma_4=\left\{ (\theta,\beta_4) : \theta\in\left[\pi/4,\pi/4+\alpha_4\right] \right\},
\end{equation}
for some $\alpha_4,\beta_4$ sufficiently small. The critical manifold $\mathcal{S}^1_0$ in system \eqref{eq:wr_system} is given simply by $\{\theta=\pi/4\}$, and due to Proposition~\ref{PROP:hyperbolicityChart2} any compact submanifold bounded away from $(0,0)$ is normally hyperbolic. By Fenichel's theorem, for $\epsilon>0$ sufficiently small, the critical manifold persists smoothly as stated in the next proposition. 
\begin{proposition}
    \label{PROP:S1_approximation}
    Any slow manifold $\mathcal{S}^1_{\epsilon}$ is locally the graph of a function \[\theta_{\epsilon}(r)=\frac{\pi}{4}+\frac{r}{\sqrt{2}}\epsilon^{3/2}+
    o\left(\epsilon^{3/2}\right).\]
    In particular, $\theta_{\epsilon}(r)=\pi/4 + \Theta(\epsilon^{3/2})$.
\end{proposition}
\begin{proof}
    See Appendix~\ref{APPENDIX:approx_S1}.
\end{proof}
\begin{remark}
\label{RMK:timescale2}
    Similarly to Remark~\ref{RMK:timescale1}, Proposition~\ref{PROP:S1_approximation} implies that the time scale near $\mathcal{S}^1_{\epsilon}$ is of order $\Theta\left(\epsilon^{-3/2}\right)$. This, in comparison to the case near $\mathcal{S}^2_{\epsilon}$, is in fact a superslow regime, see Remark~\ref{RMK:timescale1}.
\end{remark}

The third transition map will be constructed similarly to $\Pi_{12}$ and is established in the next lemma.
\begin{lemma}
    \label{LEMMA:transition3}
        For each $\Sigma_4$ as in \eqref{eq:Section4} there exists a section $\Sigma_3$ such that the transition map $\Pi_{34}:\Sigma_3\rightarrow\Sigma_4$ is well defined for every $\epsilon$ sufficiently small. The point $\mathcal{F}(\rho_*)\cap \Sigma_3$ is mapped to $(\pi/4,\beta_4)+\Theta(\epsilon^{3/2})$. Moreover, the transition map $\Pi_{34}$ is a contraction with a contraction rate of order $\mathcal{O}\left(e^{-c_3/\epsilon}\right)$ for some $c_3>0$.
\end{lemma}
\begin{proof} Since $\mathcal{F}_0$, being the stable foliation of $\mathcal{S}_0^1$, is transversal to $\Sigma_3$, then so is the stable foliation $\mathcal{F}_\epsilon^1$ of the slow manifold $\mathcal{S}_\epsilon^1$ for sufficiently small $\epsilon$. The result follows analogously to Lemma~\ref{LEMMA:Transition1}. \end{proof}

\subsection{The transition map \texorpdfstring{$\Pi_{41}$}{P41} -- passage through the nonhyperbolic point}

We construct now the remaining transition map $\Pi_{41}$. Since the equilibrium point $P_0$ is nonhyperbolic, cf. Proposition~\ref{PROP:hyperbolicityChart2}, we perform a geometric blow-up transformation in order to define $\Pi_{41}$ as a composition of intermediate transition maps near $P_0$, similarly to the methods used in \cite{Szmolyan09}.

For simplicity, we take $\theta=\pi/4+\omega$, where $\omega\in[0,\pi/4]$, and $\epsilon=\varepsilon^2$ in order to perform Taylor expansions near $\varepsilon=0$.
System \eqref{eq:MainBrusselator} expressed in the $\omega$-variable reads as (see also Appendix~\ref{APPENDIX:approx_S1})
\begin{equation}
\label{eq:wr_system}
\begin{array}{rcl}
\dot{\omega} &=& -ar^2\sin\omega (\sin\omega+\cos\omega)+
\sin^2\omega (\sin\omega+\cos\omega)(\cos\omega-\sin\omega)\\
& & \qquad +\varepsilon^2\left[
-r^2\sin\omega(\cos\omega-\sin\omega) +
\varepsilon(a/\sqrt{2}) r^3(\cos\omega-\sin\omega)
\right],\\
\\
\dot{r}&=& -ar^3\sin\omega (\cos\omega-\sin\omega) +r\sin^2\omega (\cos\omega-\sin\omega)^2 \\
& &\qquad +\varepsilon^2\left[
r^3\sin\omega (\sin\omega +\cos\omega) -
\varepsilon(a/\sqrt{2})r^4 (\sin\omega +\cos\omega)
\right].
\end{array}
\end{equation}

\begin{remark}
    \label{RMK:S1_omegaCoordinate}
    Due to Proposition~\ref{PROP:S1_approximation}, any slow manifold $S^1_{\varepsilon}\equiv S^1_{\epsilon^{1/2}}$ is locally the graph of a function
    \[
        \omega_\varepsilon(r)=\varepsilon^3 \frac{r}{\sqrt{2}}+o\left(\varepsilon^3\right).
    \]
    See also Appendix~\ref{APPENDIX:approx_S1}.
\end{remark}

Since the equilibrium point $P_0=(0,0)$ is fully nonhyperbolic, 
we introduce the quasi-homogeneous blow-up transformation
\begin{equation}
    \label{eq:blowUp_Second}
        \begin{array}{rcl}
            \omega &=& \eta^6 \Bar{\omega}\\
            r &=& \eta^3 \Bar{r}\\
            \varepsilon &=& \eta\Bar{\varepsilon},
        \end{array}
\end{equation}
where $\bar{\omega}^2+\bar{r}^2+\bar{\varepsilon}^{2}=1$ and $\eta\in[0,\eta_0]$ for some $\eta_0\ll 1$. This corresponds with a transformation  $\Phi:B\times[0,\eta_0]\rightarrow \mathbb{R}^3$. We study the system in the blown-up coordinates by means of three local charts, namely the entry chart $K_1$, the scaling chart $K_2$, and the exit chart $K_3$, defined respectively by setting $\bar{r}=1$, $\bar{\epsilon}=1$, and $\bar{\omega}=1$. The coordinates in each chart $K_i$, $i=1,2,3$, are thus given by 
\begin{equation}
\label{eq:coordinates_charts}
    \begin{array}{cccc}
       K_1: &  \omega=\eta_1^6\omega_1, & r=\eta_1^3, & \varepsilon=\eta_1\varepsilon_1,\\
       K_2: &  \omega=\eta_2^6\omega_2, & r=\eta_2^3r_2, & \varepsilon=\eta_2, \\
        K_3: & \omega=\eta_3^6, & r=\eta_3^3 r_3, & \varepsilon=\eta_3\varepsilon_3.
    \end{array}
\end{equation}
The change of coordinates from $K_1$ to $K_2$, and from $K_2$ to $K_3$, are thus given respectively by the functions $T_{12}:K_1\rightarrow K_2$ and $T_{23}:K_2\rightarrow K_3$ defined as
\begin{equation}
    \label{eq:coordsChange_second}
    \begin{array}{c}
        (\omega_2,r_2,\eta_2)= T_{12}(\omega_1,\eta_1,\varepsilon_1):=
        \left( 
        \frac{\omega_1}{\varepsilon_1^6}, 
        \frac{1}{\varepsilon_1^3},\eta_1\varepsilon_1
        \right), \\
        \\
        (\eta_3,r_3,\varepsilon_3)=T_{23}(\omega_2,r_2,\eta_2):=
        \left(\eta_2\omega_2^{1/6}, \frac{r_2}{\omega_2^{1/2}}, \frac{1}{\omega_2^{1/6}}  \right).
    \end{array}
\end{equation}
In the following, we construct a transition map $\Pi_{41}:\Sigma_4\rightarrow\Sigma_1$ which will be related to a composition of three intermediate transition maps $\Pi_{\TU{in}}$, $\Pi_{\TU{tran}}$, and $\Pi_{\TU{out}}$, defined on the charts $K_1$, $K_2$, and $K_3$, respectively. In this way, the transition map $\Pi_{41}$ is defined as
\[
    \Pi_{41}= \Phi\circ \Pi_{\TU{out}} \circ T_{23}\circ \Pi_{\TU{tran}} \circ T_{12} \circ \Pi_{\TU{in}}\circ \Phi^{-1}.
\]
A summary of the intermediate steps is portrayed in Figure~\ref{fig:blow_up} below.

\begin{figure}[ht]
     \centering
      \begin{overpic}[width=.8\linewidth]{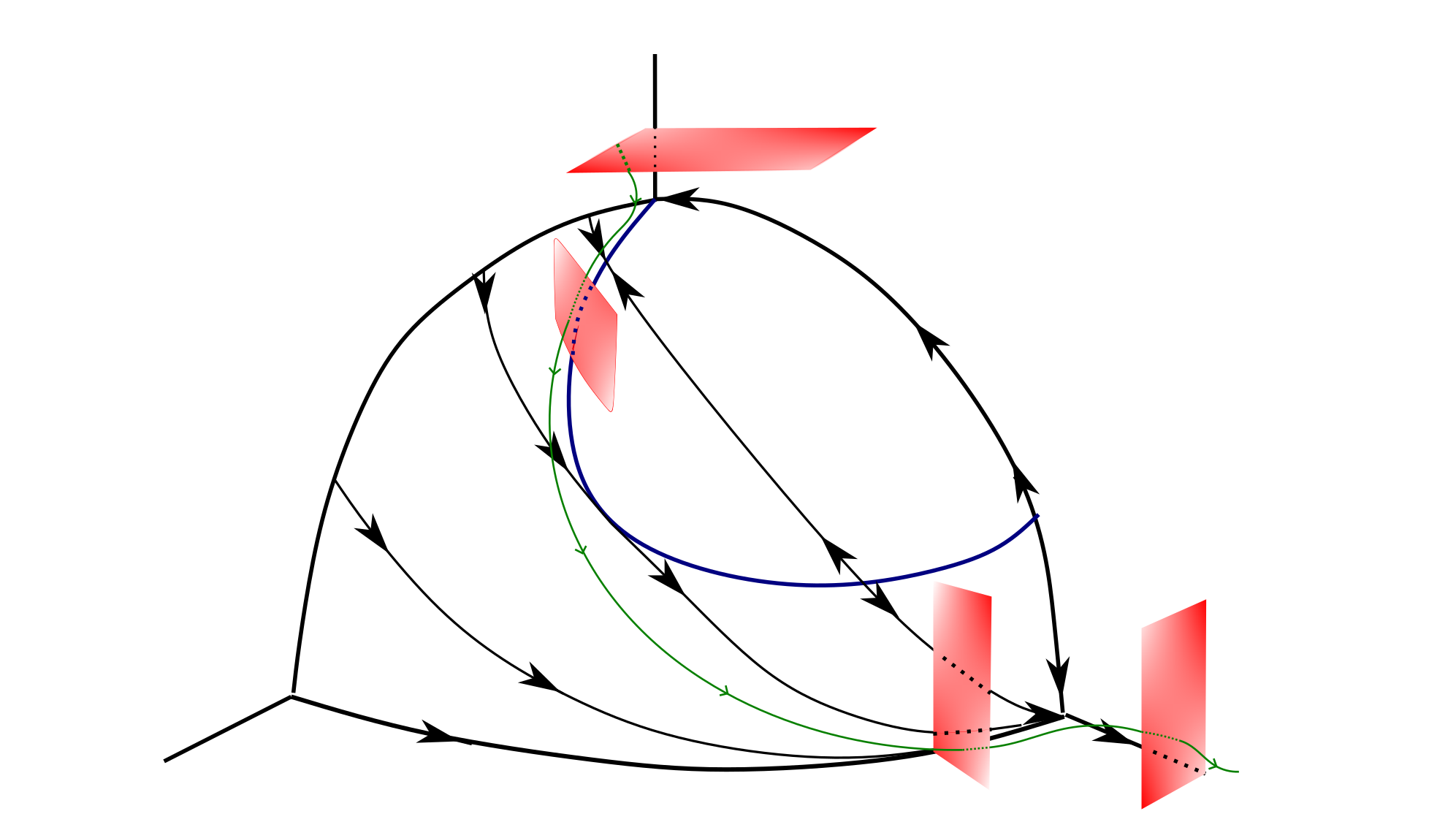}
      \put(50,49){\scriptsize $\Sigma_4$}
      \put(44,53){ $\bar{r}$}
      \put(8,3){ $\bar{\varepsilon}$}
      \put(86,3){$\bar{\omega}$}
      \put(43,30){\scriptsize $\Sigma_{in}$}
      \put(65,1){\scriptsize $\Sigma_{out}$}
      \put(79,16){\scriptsize $\Sigma_1$}
      \end{overpic}
       \hfill
       \caption[Dynamics in second chart]{Scheme of the transition map $\Pi_{41}$, via intermediate transition maps through the sections $\Sigma_{\TU{in}}$ and $\Sigma_{\TU{out}}$ after the equilibrium point $P_0$ is blown-up to a sphere $\{\bar{\omega}^2+\bar{r}^2+\bar{\varepsilon}^2=1\}\times\{0\}$, where the nearby dynamics is revealed after performing appropriate desingularisation. In green, a prototypical orbit passing through all sections is sketched, for $\varepsilon > 0$ sufficiently small.}
        \label{fig:blow_up}
\end{figure}
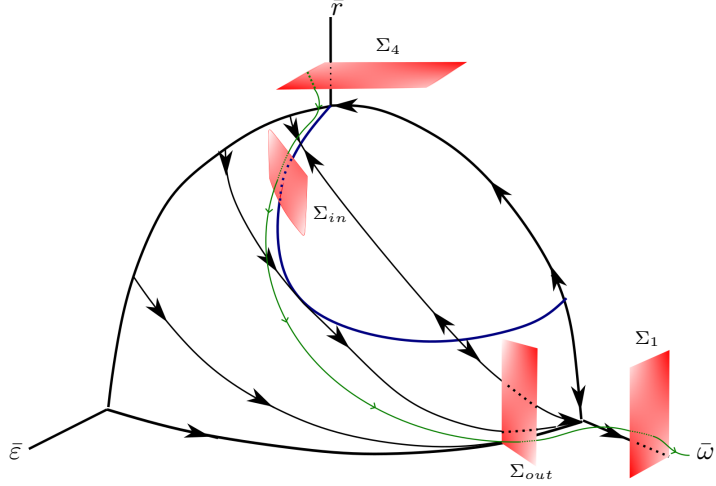

\subsubsection{Dynamics in the entry chart}
From \eqref{eq:wr_system} and \eqref{eq:coordinates_charts}, it follows that the dynamics in chart $K_1$ is given by
\[
         \dot{\omega_1}= \frac{\dot{\omega}}{\eta_1^6}-\frac{6\omega_1\dot{\eta_1}}{\eta_1}, \qquad
         \dot{\eta_1} = \frac{\dot{r}}{3\eta_1^2},
         \qquad
         \dot{\varepsilon_1} =-\frac{\varepsilon_1\dot{\eta_1}}{\eta_1}.
\]
Therefore, by using $\sin\omega=\omega H(\omega)$, where 
\[H(\omega)=1-\frac{\omega^2}{6}+\frac{\omega^4}{5!}+o(\omega^5),\]
and dividing the resulting vector field by $\eta_1^6$, the ODEs in chart $K_1$ read as 
\begin{equation}
\begin{array}{rl}
\label{eq:blowchart1}
    \omega_1'= & g(\omega_1,\eta_1,\varepsilon_1)-2 \omega_1\eta_1^6h(\omega_1,\eta_1,\varepsilon_1)
     \\
     \eta_1'= & \frac{\eta_1^7}{3}h(\omega_1,\eta_1,\varepsilon_1)
     \\
     \varepsilon_1'= & -\frac{\varepsilon_1 \eta_1^6}{3} h(\omega_1,\eta_1,\varepsilon_1),
\end{array}
\end{equation}
where \begin{align*}
    g(\omega,\eta,\varepsilon)= & -a\omega H(\eta^6\omega)(\sin\eta^6\omega+\cos\eta^6\omega)+
    \omega^2 H^2(\eta^6\omega)(\sin\eta^6\omega+\cos\eta^6\omega) (\cos\eta^6\omega-\sin\eta^6\omega) \\
    & -\eta^2\omega\varepsilon^2 H(\eta^6\omega)(\cos\eta^6\omega-\sin\eta^6\omega) +
    \frac{a}{\sqrt{2}}\varepsilon^3(\cos\eta^6\omega-\sin\eta^6\omega),\\
    h(\omega,\eta,\varepsilon)= & -a\omega H(\eta^6\omega)(\cos\eta^6\omega-\sin\eta^6\omega)+ \omega^2 H^2(\eta^6\omega)(\cos\eta^6\omega-\sin\eta^6\omega)^2 + \\
   & \eta^2\omega^2\varepsilon^2 H(\eta^6\omega)(\sin\eta^6\omega+\cos \eta^6\omega) +
    \frac{a}{\sqrt{2}}\varepsilon^3.
\end{align*}
Here, we write $\cdot':= d/d\tilde{\tau}_1$ for the corresponding rescaled time variable $\tilde{\tau}_1$. Note that the quantity $\eta_1\varepsilon_1$ remains constant along orbits.
Firstly, we analyse how system \eqref{eq:blowchart1} behaves near the invariant planes $\{\eta_1=0\}$ and $\{\varepsilon_1=0\}$. In the former case, \eqref{eq:blowchart1} reads as
\begin{equation}
\label{eq:ODE1}
\begin{array}{rl}
     \omega_1'= & -a\omega_1+\omega_1^2+\frac{a}{\sqrt{2}}\epsilon_1^3, \\
     \varepsilon_1'= & \ 0.
    \end{array}
\end{equation}
It follows that the dynamics on each line $\{\eta_1=0, \varepsilon_1=c\}$, for every $c \geq 0$, is invariant. Let $\mathcal{N}$ be the graph of
\begin{equation}
\label{eq:branch1_equilibria}
    \varepsilon_1=\left(\frac{\sqrt{2}\cdot\omega_1(a-\omega_1)}{a}\right)^{1/3},
\end{equation}
which is a curve of equilibria for \eqref{eq:blowchart1} defined only for $\omega_1\in[0,a]$, see Figure~\ref{fig:Entry_chart}. In fact, system \eqref{eq:ODE1} seen as a one-dimensional ODE parametrized by $\epsilon_1=c$ admits a fold bifurcation at $(\omega_1^*,0,\varepsilon_1^*)=\left(\frac{a}{2},0,(\frac{a}{2\sqrt{2}})^{1/3}\right)$. 
Equivalently, since \eqref{eq:branch1_equilibria} is locally invertible away from the fold point $(\omega_1^*,0,\varepsilon_1^*)$, we can define the curve $\mathcal{N}$ given by the graphs of
\begin{equation}
    \label{eq:brances_inverse}
        \omega_1^{\pm}=\frac{a-\sqrt{a^2-2\sqrt{2}a\varepsilon_1^3}}{2}.
\end{equation}

We aim to construct a section $\Sigma_{\TU{in}}$ transverse to the graph of \eqref{eq:branch1_equilibria} and sufficiently close to the fold point $(\omega_1^*,0,\varepsilon_1^*)$ such that there is a well-defined transition map $\Pi_{\TU{in}}:\tilde{\Sigma}_4\rightarrow\Sigma_{\TU{in}}$. 

In the invariant plane given by $\varepsilon_1=0$, the dynamics is given by
\begin{align*}
    \omega_1' =& (-a\omega_1+\omega_1^2) \left(1+\mathcal{O} (\eta_1^6\omega_1)\right) +2\eta_1^6(a\omega_1^2-\omega_1^2) \left(1+\mathcal{O}(\eta_1^6\omega_1)\right),\\
    \eta_1' =& -\frac{\eta_1^7\omega_1(a-\omega_1)}{3} \left(1+ \mathcal{O}(\eta_1^6\omega_1)\right).
\end{align*}
It is straightforward to see that the vertical axis $\{\omega_1=\epsilon_1=0\}$ consists entirely of equilibria for system \eqref{eq:blowchart1}. In particular, the origin is a partially hyperbolic equilibrium as elaborated in the next proposition.
\begin{proposition}
    \label{PROP:centremanifold_chart1}
    The following statements hold for system \eqref{eq:blowchart1}:
    \begin{enumerate}
        \item The linearisation of \eqref{eq:blowchart1} at the equilibrium $Q_0=(0,0,0)$ has a negative simple eigenvalue $-a$, and zero as a double eigenvalue. Their corresponding eigenspaces are $E_{-a}=span\{(1,0,0)^{\top}\}$ and $E_0=span\{(0,1,0)^{\top}, (0,0,1)^{\top}\}$, respectively.
        \item There exists a 2-dimensional exponentially attracting centre manifold $W_{loc}^c(Q_0)$ at the origin $Q_0$, containing the vertical axis and the graph of $\omega_1^-$ as given in \eqref{eq:brances_inverse} for $\varepsilon_1 < \varepsilon_1^*$. Hence, $W_{loc}^c(Q_0)$ is defined locally as the graph of the function
        \begin{equation}
            \label{eq:graph_centremanifold}
        \omega_1=V(\eta_1,\varepsilon_1)=\omega_1^-(\varepsilon_1)+\mathcal{O}(\eta).
        \end{equation}
    \end{enumerate}
\end{proposition}
\begin{proof} Assertion (1) follows directly from the expression of the Jacobian matrix at $(0,0,0)$, which is
\[
    \left[ 
    \begin{array}{ccc}
    -a & 0 & 0\\
    0 & 0 & 0 \\
    0 & 0 & 0
    \end{array}\right].
\]

The existence and attractiveness of $W^c_{loc}(Q_0)$ is given by the \textit{centre manifold theorem}. Recall that the centre manifold is given as the graph of a function $\omega_1=V(\eta_1,\varepsilon_1)$, which is $\mathcal{O}\left(\left\Vert (\eta_1,\epsilon_1) \right\Vert^2\right)$ for being tangent to $E_0$. Moreover, since it contains the vertical axis, it is given as the graph of a function $\omega_1=\varepsilon_1 \tilde{V}(\eta_1,\varepsilon_1)$, where $\tilde{V}(\eta_1,\varepsilon_1)=\mathcal{O}\left(\left\Vert (\eta_1,\epsilon_1) \right\Vert\right)$. 

Since $Q_0\in \mathcal{N}\cap \{\omega_1=\varepsilon_1=0\}$, its centre manifold $W_{loc}^c$ contains both curves. Hence, we obtain \eqref{eq:graph_centremanifold} for the first term in the Taylor expansion of $V$ with respect to $\eta$.
\end{proof}

Since $\Sigma_4$ does not depend on $\epsilon$, c.f. \eqref{eq:Section4}, we lift it into chart $K_1$. Concretely, we consider
\begin{equation}
    \label{eq:section 4_ chart1}
    \tilde{\Sigma}_4=\left\{ \left(\omega_1,\tilde{\beta}_4,\epsilon_1\right) : \omega_1\in[0,\tilde{\alpha}_4], \ \varepsilon_1\in[0,\tilde{\gamma}_4]\right\},
\end{equation}
where $\tilde{\alpha_4}=\alpha_4/\beta_4^2$, $\tilde{\beta}_4=\beta_4^{1/3}$, and $\tilde{\gamma_4}$ is sufficiently small.

Consider the intermediary section 
\begin{equation} 
\label{eq:Section_in}
\Sigma_{\TU{in}}=\left\{ 
    \left( \omega_1, \eta_1,\gamma_{\TU{in}} \right) :
    \vert\omega_1-V(0,\gamma_{\TU{in}})\vert\leq \alpha_{\TU{in}}, \ \eta_1\in[0,\beta_{\TU{in}}]
    \right\},
\end{equation}
for sufficiently small positive constant $\alpha_{\TU{in}}$ (where, in particular, $V(0,\gamma_{\TU{in}})-\alpha_{\TU{in}}>0$), and $\gamma_{\TU{in}}$ below but arbitrarily close to $\epsilon_1^*$, see Figure~\ref{fig:Entry_chart}. 

We consider in more detail the dynamics on $W_{loc}^c(Q_0)$, which are depicted in Figure~\ref{fig:Entry_chart} (b). In particular, we show that $\eta_1$ is monotonically decreasing as stated in the following proposition. 
\begin{proposition}
    \label{PROP:eta1_decreasing}
        When restricted to the centre manifold, $\eta_1$ is monotonically decreasing whenever $\eta_1(0), \varepsilon_1(0)\neq 0$ and $\varepsilon_1<\gamma_{\TU{in}}$
\end{proposition}
\begin{proof}
We use the expression $\omega_1=V(\eta_1,\varepsilon_1)$ as given in Proposition~\ref{PROP:centremanifold_chart1} in \eqref{eq:blowchart1}, for which we first notice that
\[
    \eta_1'=\frac{\eta_1^7}{3}\left[ -a\omega + \omega^2-\frac{a\varepsilon^3}{\sqrt{2}} + \mathcal{O}(\eta_1^2)\right].
\]
Substituting $\omega_1=V(\eta_1,\varepsilon_1)$ in the equation above yields
\[
    \eta_1'=\frac{\eta_1^7}{3}\left[ -\frac{a\sqrt{a^2-2\sqrt{2}a\varepsilon_1^3}}{2}-2\sqrt{2}\varepsilon_1^3+\mathcal{O}(\eta_1^2) \right]<
    \frac{\eta_1^7}{3}\left[ -\frac{a\sqrt{a^2-2\sqrt{2}a\gamma_{\TU{in}}^3}}{2}+\mathcal{O}(\eta_1^2) \right] <0,
\]
for every $\eta_1$ sufficiently small. In particular, if $\tilde{\beta}_4$ is small enough, $\eta_1$ is monotonically decreasing.
\end{proof}

Since the quantity $\eta_1\varepsilon_1$ remains constant, it follows from Proposition~\ref{PROP:eta1_decreasing} that $\varepsilon_1$ is monotonically increasing. Therefore, the transition map $\Pi_{\TU{in}}$ is well defined when restricted to $W_{loc}^c\cap\tilde{\Sigma}_4$, where we define $\Pi_{in}\left(0,\tilde{\beta}_4,0\right)=\left(0,V(0,\gamma_{\TU{in}}),\gamma_{\TU{in}})\right)$. The extension to the whole section $\Sigma_4$ is given by the attractiveness of the centre manifold, as summarised in the following proposition, which now follows directly.
\begin{proposition}
    \label{LEMMA:entrance_transition}
    For each $\Sigma_{\TU{in
    }}$ as in \eqref{eq:Section_in}, there is $\tilde{\Sigma}_4$ such that the transition map $\Pi_{\TU{in}}:\tilde{\Sigma}_4\rightarrow \Sigma_{\TU{in}}$ is well-defined. Moreover, for each $\varepsilon_1$ fixed, the set 
    \[\left\{ \Pi_{\TU{in}}\left(\omega_1, \tilde{\beta}_4,\varepsilon_1 \right): \omega_1\in[0,\tilde{\alpha}_4]\right\} \subset 
    \left\{ \eta_1=\frac{\varepsilon_1\tilde{\beta}_4}{\gamma_{\TU{in}}}, \ \varepsilon_1=\gamma_{\TU{in}} \right\}
    \]
    is a segment of length $\mathcal{O}\left(e^{-\tilde{c}_1/\varepsilon_1}\right)$ for some $\tilde{c}_1>0$, and therefore $\Pi_{\TU{in}}\left(  \tilde{\Sigma}_4\right)$ is an exponentially thin wedge around $W_{loc}^c(Q_0)\cap\Sigma_{\TU{in}}$. More precisely, for each $(\omega_1,\tilde{\beta}_4,\varepsilon_1)\in \tilde{\Sigma}_4$ we have
    \[
        \Pi_{\TU{in}}(\omega_1,\tilde{\beta}_4,\varepsilon_1)=
        \left( 
         \tilde{\omega}_1(\omega_1,\varepsilon_1) +\mathcal{O}\left(
        e^{-\tilde{c}_1/\varepsilon_1}\right) ,\frac{\varepsilon_1\tilde{\beta}_4}{\gamma_{\TU{in}}},\gamma_{\TU{in}}\right),
    \]
   where $\tilde{\omega}_1(\omega_1,\varepsilon_1)\in W_{loc}^c(Q_0)\cap \Sigma_{\TU{in}}$.
\end{proposition}

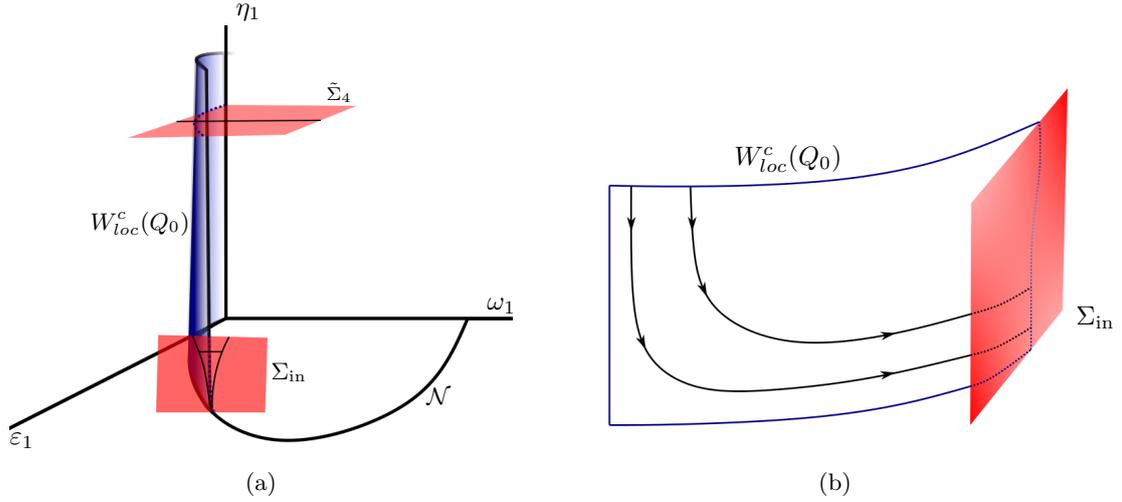
\begin{figure}[ht]
\centering
    \begin{subfigure}{0.45\textwidth}
     \centering
      \begin{overpic}[width=\linewidth]{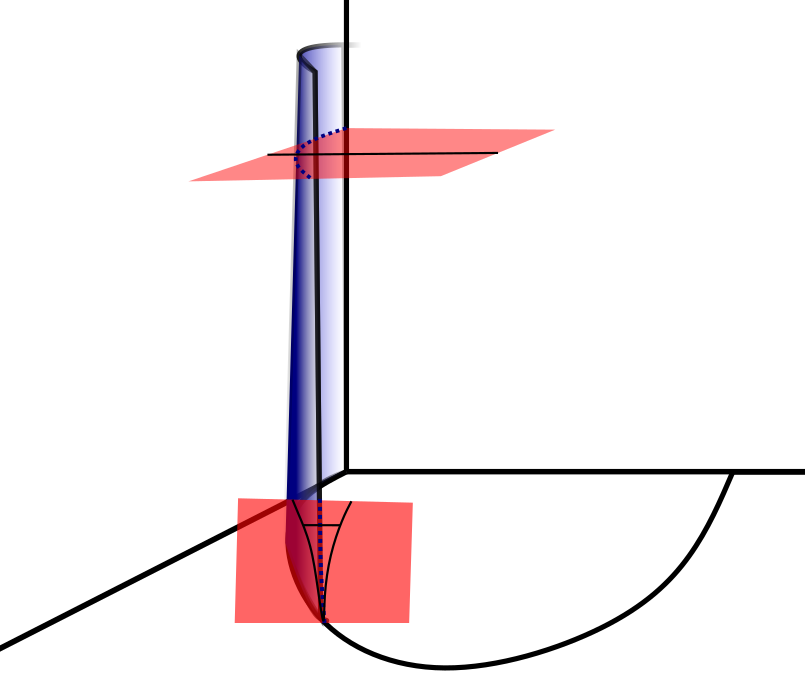}
      \put(0,2){$\varepsilon_1$}
\put(95,29){$\omega_1$}
\put(45,88){$\eta_1$}  
\put(63,71){\scriptsize$\tilde{\Sigma}_4$}
\put(52,15){\small$\Sigma_{\TU{in}}$}
\put(83,10){\small$\mathcal{N}$}
\put(16,45){\small$W_{loc}^c(Q_0)$}
      \end{overpic}
      \caption{}
    \end{subfigure}
    \hspace{2em}
        \begin{subfigure}{0.45\textwidth}
        \centering
    \begin{overpic}[width=\textwidth]{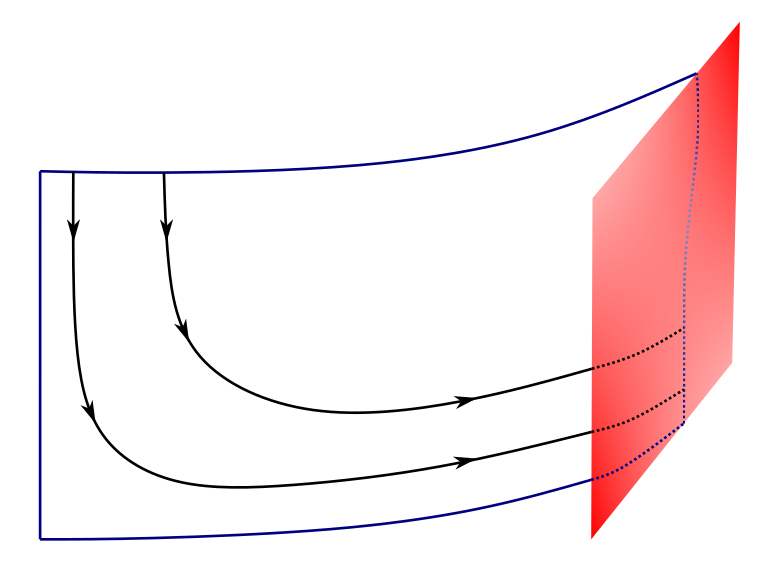}
    \put(30,58){$W_{loc}^c(Q_0)$}
    \put(98,26){$\Sigma_{\TU{in}}$}
    \end{overpic}
        \caption{}
    \end{subfigure}

    \hfill
 \caption[Dynamics in entry chart]{The transition map $\Pi_{\TU{in}}$ in the entry chart $K_1$. In (a), the sections $\tilde{\Sigma}_4$ and $\Sigma_{\TU{in}}$ are portrayed in red. In blue, a sketch of $W_{loc}^c(Q_0)$ is shown, where its intersection with $\Sigma_1,\Sigma_{\TU{in}}$ is drawn as a dotted line. The image set  $\Sigma_{\TU{in}} \cap \Pi_{\TU{in}}\left( \tilde{\Sigma}_4\right)$ is an exponentially small wedge, with vertex on the line of equilibria $\mathcal{N}\subset\{\eta_1=0\}$. In (b), a sketch of Proposition~\ref{PROP:centremanifold_chart1} is presented, where the dynamics on the centre manifold are depicted in black.}
         \label{fig:Entry_chart}
\end{figure}

\subsubsection{Dynamics in the scaling chart}
We now analyse the dynamics in chart $K_2$. Note that
\[
        \dot{\omega}_2=\frac{\dot{\omega}}{\eta_2^6}, \qquad \dot{r_2}=\frac{\dot{r}}{\eta_2^3}.
\]
Therefore, after reparametrising the time variable so that the vector field is divided by $\eta_2^6$ we obtain
\begin{equation}
    \label{blowchart_tran}
    \begin{split}
    \omega_2'= & -ar_2^2\omega_2(\sin\eta_2^6\omega_2+\cos\eta_2^6\omega_2)+ 
    \omega_2^2(\sin\eta_2^6\omega_2+\cos\eta_2^6\omega_2)(\cos\eta_2^6\omega_2-\sin\eta_2^6\omega_2)\\
    & +\frac{a}{\sqrt{2}}r_2^3(\cos\eta_2^6\omega_2-\sin \eta_2^6\omega_2)+ \eta_2^2r_2^2\omega_2(\cos\eta_2^6\omega_2-\sin\eta_2^6\omega_2) + \mathcal{O}(\eta_2^6\omega_2),\\
    r_2'= & \eta_2^6r_2\left[
    -ar_2^2\omega_2(\cos\eta_2^6\omega_2-\sin\eta_2^6\omega_2)+\omega_2^2(\cos\eta_2^6\omega_2-\sin\eta_2^6\omega_2)^2 \right]\\
    &+\eta_2^6r_2\left[-\frac{a}{\sqrt{2}}r^3(\sin\eta_2^6\omega_2+\cos\eta_2^6\omega_2)+\eta_2^2r_2^2\omega_2(\sin\eta_2^6\omega_2+\cos\eta_2^6\omega_2) 
    \right] + \mathcal{O}(\eta_2^6\omega_2)\\
    \eta_2'= & 0,\\
    \end{split}
\end{equation}
where $\cdot':=d/d\tilde{\tau}_2$ for the corresponding rescaled time variable. Observe that \eqref{blowchart_tran} is a slow-fast system in standard form. Its layer problem is given by $\eta_2=0$ and reads as
\begin{align*}
    \omega_2'= & -ar_2^2\omega_2+\omega_2^2+\frac{a}{\sqrt{2}}r_2^3, \\
    r_2'= & 0,
\end{align*}
where the set of equilibria is given by the two branches
\begin{equation}
    \label{eq:equilibria_chart2}
        \omega_2^{\pm}(r_2)=\frac{ar_2^2\pm \sqrt{a^2r_2^4-2^{3/2}ar_2^3}}{2}.
\end{equation}
By means of the change of coordinates $T_{12}$, it is straightforward to see that $\mathcal{N}$, cf. \eqref{eq:branch1_equilibria}, in chart $K_2$ is given by \eqref{eq:equilibria_chart2}. The fold point is given precisely at $(\omega_2^*,r_2^*,0)$, where $\omega^+_2(r_2)=\omega^-_2(r_2)$, so that $r_2^*=\frac{2\sqrt{2}}{a}$ and $\omega_2^*=\frac{4}{a}$. Therefore, the transition map we aim to construct is a transition through a fold point, similar to the transition in Subsection~\ref{SEC:PI_23}.

Notice that the transversal $\Sigma_{\TU{in}}$ in chart $K_2$ is given by 
\begin{equation}
\label{eq:Sigma_in_K2}
    \Sigma_{\TU{in}}=\left\{ \left( \omega_2, \tilde{\beta}_{\TU{in}}, \eta_2 \right) : 
    \vert \omega_2- \delta \vert \leq \tilde{\alpha}_{\TU{in}}, \
    \eta_2\in[0,\tilde{\gamma}_{\TU{in}}]\right\}
\end{equation}
where $\delta=V(0,\gamma_{\TU{in}})/\gamma_{\TU{in}}^6$, $\tilde{\alpha}_{\TU{in}}=\alpha_{\TU{in}}/\gamma_{\TU{in}}^6$, $\tilde{\beta}_{\TU{in}}=1/\gamma_{\TU{in}}^3$ and $\tilde{\gamma}_{\TU{in}}=\beta_{\TU{in}}\gamma_{\TU{in}}$. Here $\tilde{\alpha}_{\TU{in}}, \tilde{\gamma}_{\TU{in}}$ are sufficiently small. Consider the transverse section
\begin{equation}
\label{eq:Sigma_out}
    \Sigma_{\TU{out}}=\{(\omega_2^*+\alpha_{\TU{out}},r_2,\eta_2) : r_2\in[0,r_2^*+ \beta_{\TU{out}}], \; \eta_2\in[0,\gamma_{\TU{out}}]\}
\end{equation}
for some positive constants $\alpha_{\TU{out}}>0$, and sufficiently small $\beta_{\TU{out}},\gamma_{\TU{out}}>0$. A transition map $\Pi_{\TU{tran}}:\Sigma_{\TU{in}}\rightarrow\Sigma_{\TU{out}}$ is given in the next proposition.
\begin{proposition}
    \label{LEMMA: transit_fold_blowup}
For each transversal $\Sigma_{\TU{out}}$ as in \eqref{eq:Sigma_out} , there exists $\Sigma_{\TU{in}}$ as in \eqref{eq:Sigma_in_K2} such that the transition map $\Pi_{\TU{tran}}:\Sigma_{\TU{in}}\rightarrow\Sigma_{\TU{out}}$ is well defined. For each $\eta_2=\eta_{2,0}$  fixed,  the set 
\[
   \left\{ \Pi_{\TU{tran}}\left(\omega_2,\tilde{\beta}_{\TU{in}},\eta_{2,0} \right) :  \vert\omega_2-\delta\vert \leq \tilde{\alpha}_{\TU{in}}  \right\} \subset
   \left\{ \omega_2= \omega_2^*+\alpha_{\TU{out}}, \ \eta_2=\eta_{2,0} \right\}
\]
is a segment of length $\mathcal{O}\left(e^{-\tilde{c}_2/\eta_{2,0}^6}\right)$, for some $\tilde{c}_2>0$, around $(\omega_2^*+\alpha_{\TU{out}}, r_2^*+\Theta\left( \eta_{2,0}^4 \right),\eta_{2,0})$. Hence, for every $\left(\omega_2,\tilde{\beta}_{\TU{in}} ,\eta_2 \right)\in\Sigma_{\TU{in}}$,
\[
    \Pi_{\TU{tran}}\left( \omega_2,\tilde{\beta}_{\TU{in}},\eta_2 \right)=
    \left( \omega_2^*+\alpha_{\TU{out}},r_2^*+\Theta \left( \eta_2^4\right),\eta_2 \right)
\]
\end{proposition}
\begin{proof}
    Consider the functions
\begin{align*}
f(\omega,r)= & -ar^2\omega+\omega^2+\frac{a}{\sqrt{2}}r^3,    \\
g(\omega,r)= &\left(-ar^2\omega+\omega^2r-\frac{a}{\sqrt{2}}r^3\right) r,
\end{align*}
so that, by renaming the small parameter $\xi=\eta_2^6$, \eqref{blowchart_tran} reads as
\begin{align*}
\omega_2'= & f(\omega_2,r_2) + \mathcal{O}(\xi^{1/3}), \\
r_2'= &\xi \left[ g(\omega_2,r_2) +\mathcal{O}(\xi^{1/3}) \right].
\end{align*}
Since $(\omega_2^*,r_2^*,0)$ is a fold point, it satisfies
\[
    f(\omega_2^*,r_2^*)=\partial_{\omega_2}f(\omega_2^*,r_2^*)=0.
\]
The main statement is now a direct application of \cite[Theorem 2.1 and Remark 2.11]{Krupa01}, where the following aditional nondegeneracy conditions are required:
\[ 
\partial^2_{\omega\omega}f(\omega_2^*,r_2^*)\neq 0, \qquad \partial_rf(\omega_2^*,r_2^*)\neq 0, \qquad g(\omega_2^*,r_2^*)\neq 0. 
\]
This is indeed the case for our situation since 
\[
    \partial^2_{\omega\omega}f(\omega_2^*,r_2^*)= \frac{8}{a},
    \quad \partial_rf(\omega_2^*,r_2^*)=-\frac{8}{\sqrt{2}a}, \qquad g(\omega_2^*,r_2^*)=-\frac{128}{a\sqrt{2}}.
\]
It follows that the perturbation order is $\Theta\left( \xi^{2/3} \right)$ and the length of the interval is $\mathcal{O}\left( e^{-\tilde{c}_2/\xi} \right)$; hence, by substituting back $\xi=\eta_2^6$, we can conclude the result.
\end{proof}

\begin{remark}
    \label{REMARK: alpha_big} Note that we may choose $\alpha_{\TU{out}}$ arbitrarily large. This fact will be used in the next subsection.
\end{remark}

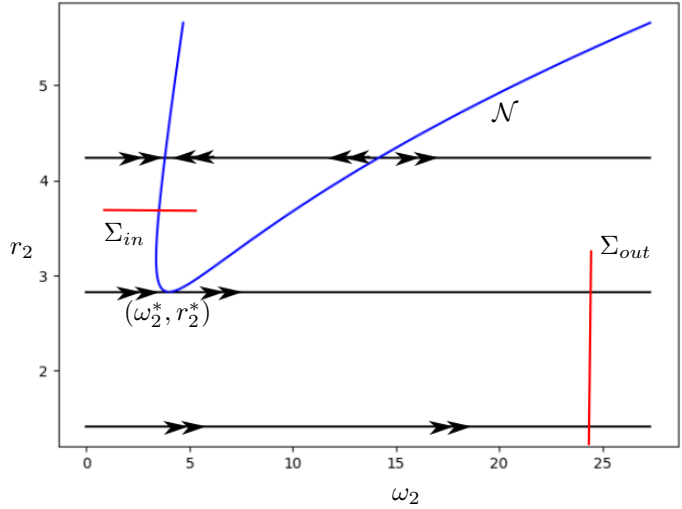
\begin{figure}[ht]
     \centering      \begin{overpic}[width=.6\linewidth]{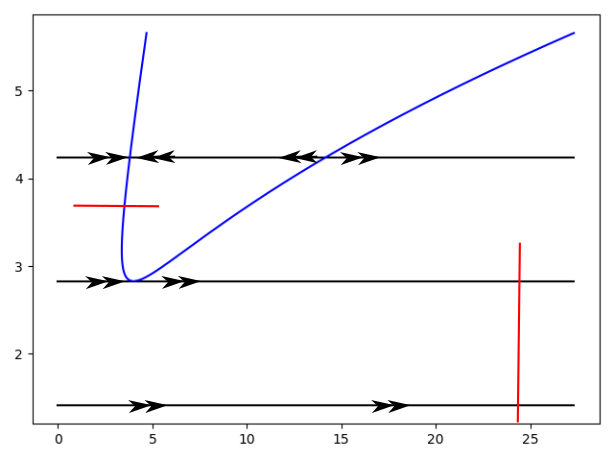}
     \put(55,-2){$\omega_2$}
     \put(-2,35){$r_2$}
     \put(12,37){$\Sigma_{in}$}
     \put(70,55){$\mathcal{N}$}
     \put(86,35){$\Sigma_{out}$}
     \put(15,25){$(\omega_2^*,r_2^*)$}
     \end{overpic}  
     \hfill
    \caption[Dynamics in scaling chart]{Transition through the regular fold point $(\omega_2^*,r_2^*)$ in the slow-fast system in standard form \eqref{blowchart_tran}, in the scaling chart $K_2$. The critical manifold $\mathcal{N}$ is portrayed in blue, and the fast fibration in black. Double arrows indicate the direction of the dynamics of the layer problem. The sections $\Sigma_{\TU{in}}$ and $\Sigma_{\TU{out}}$ are presented in red. Notice that $\Sigma_{\TU{out}}$ can be chosen arbitrarily far away from the fold point.}
        \label{fig:Scaling_chart}
\end{figure}

\subsubsection{Dynamics in the exit chart}
We now derive an ODE system for the dynamics in the chart $K_3$. From \eqref{eq:coordinates_charts}, it follows that
\[
\dot{\eta_3}=\frac{\dot{\omega}}{6\eta_3^3}, \qquad 
\dot{r_3}=\frac{\dot{r}}{\eta_3^3}-\frac{3r_3\dot{\eta_3}}{\eta_3}, \qquad \dot{\epsilon_3}=-\frac{\epsilon_3 \dot{\eta}_3}{\eta_3}.
\]
Therefore, upon dividing the vector field by $\eta_3^6$ and denoting $\cdot':=d/d\tilde{\tau}_3$ for the resulting time variable, the dynamics in $K_3$ is governed by the ODE system
\begin{equation}
    \label{eq:blowchart2}
    \begin{split}
    \eta_3'= & \frac{\eta_3}{6}
    \left( \begin{split}
    -ar_3^2(\sin\eta_3^6+\cos\eta_3^6)H(\eta_3^6)+
    (\sin\eta_3^6+\cos\eta_3^6)(\cos\eta_3^6-\sin\eta_3^6)H^2(\eta_3^6)\\
 +\frac{a}{\sqrt{2}}\varepsilon_3^3r_3^3(\cos\eta_3^6-\sin\eta_3^6)-\eta_3^2\varepsilon_3^2r_3^2(\sin\eta_3^6+\cos\eta_3^6)H(\eta_3^6) \end{split}\right) \\
    r_3' = & \frac{r_3}{2}\left(\begin{split}
    ar_3^2(\sin\eta_3^6+\cos\eta_3^6)H(\eta_3^6)-(\sin\eta_3^6+\cos\eta_3^6)(\cos\eta_3^6-\sin\eta_3^6)H^2(\eta_3^6)\\
    -\frac{a}{\sqrt{2}}\varepsilon_3^3r_3^3(\cos\eta_3^6-\sin\eta_3^6) +\eta_3^2\varepsilon_3^2r_3^2(\sin\eta_3^6+\cos\eta_3^6)H(\eta_3^6)
    \end{split}\right) \\
    & \qquad + \eta_3^{6}r_3 \left( \begin{split} 
    -ar_3^2(\cos\eta_3^6-\sin\eta_3^6)H(\eta_3^6)+(\cos\eta_3^6-\sin\eta_3^6)^2H^2(\eta_3^6) \\
    -\frac{a}{\sqrt{2}}\epsilon_3^3r_3^3(\sin\eta_3^6+\cos\eta_3^6) + \eta_3^2r_3^2\epsilon_3^2(\sin\eta_3^6+\cos\eta_3^6) H(\eta_3^6)
    \end{split}\right)\\
    \varepsilon_3' = & \frac{\varepsilon_3}{6}\left( \begin{split}
    ar_3^2(\sin\eta_3^6+\cos\eta_3^6)H(\eta_3^6) -(\sin\eta_3^6+\cos\eta_3^6)(\cos\eta_3^6-\sin\eta_3^6)H^2(\eta_3^6)\\
    -\frac{a}{\sqrt{2}}\varepsilon_3^3r_3^3(\cos\eta_3^6-\sin\eta_3^6) +\eta_3^2\varepsilon_3^2r_3^2(\sin\eta_3^6+\cos\eta_3^6)H(\eta_3^6) 
    \end{split}\right)
\end{split}.
\end{equation}

By means of the change of charts $T_{23}$, see \eqref{eq:coordsChange_second}, the section $\Sigma_{out}$ in chart $K_3$ is expressed as
\begin{equation}
\label{eq:Sigma_out_K3}
    \Sigma_{\TU{out}}=\left\{
    (\eta_3,r_3,\tilde{\gamma}_{\TU{out}}) : \eta_3\in[0,\tilde{\alpha}_{\TU{out}}], \ 
    r_3 \in [0,\tilde{\beta}_{\TU{out}}]\right\},
\end{equation}
where $\tilde{\beta}_{\TU{out}}=\frac{r_2^*+\beta_{\TU{out}}}{(\omega_2^*+\alpha)^{1/6}}$, $\tilde{\gamma}_{\TU{out}}=(\omega_2^*+\alpha_{\TU{out}})^{-1/6}$, $\tilde{\alpha}_{\TU{out}}=\gamma_{\TU{out}}(\omega_2^*+\alpha_{\TU{out}})^{1/6}$. Due to the arbitrary choice of $\alpha_{\TU{out}}$ (see Remark~\ref{REMARK: alpha_big}), without loss of generality one can take $\tilde{\alpha}_{\TU{out}},\tilde{\beta}_{\TU{out}},\tilde{\gamma}_{\TU{out}}$ arbitrarily small.

For the analysis of \eqref{eq:blowchart2}, recall that the quantity $\eta_3\varepsilon_3$ remains constant. Observe that each plane $\{\eta_3=0\}$, $\{\varepsilon_3=0\}$, and $\{r_3=0\}$, is invariant. In the first case, the flow is defined by the  system
\begin{equation}
\label{eq:invariantplane_chartOut}
\begin{split}
    r_3'=&\frac{r_3}{2}\left( ar_3^2-1-\frac{a}{\sqrt{2}}\varepsilon_3^3r_3^3 \right) \\
    \varepsilon_3'=& \frac{\varepsilon_3}{6}\left( ar_3^2-1-\frac{a}{\sqrt{2}}\varepsilon_3^3r_3^3 \right).
\end{split}
\end{equation}
A curve of equilibria is given by the equation
\begin{equation}
\label{eq:equi_invplane1}
    ar_3^2-1-\frac{a}{\sqrt{2}}\varepsilon_3^3r_3^3=0,
\end{equation}
which corresponds to the curve $\mathcal{N}$ in chart $K_2$, c.f.\eqref{eq:equilibria_chart2}, by means of the change of coordinates $T_{23}$ in \eqref{eq:coordsChange_second}. In particular, the fold point $(\omega_2^*,r_2^*,0)$ is mapped to 
\[
(0,r_3^*,\varepsilon_3^*)=\left( 0, \frac{r_2^*}{\sqrt{\omega_2^*}},\frac{1}{\sqrt[6]{w_2^*}} \right).
\]
The full phase portrait in the plane $\{\eta_3=0\}$ is given in Figure~\ref{fig:Exit_chart_planes} (a). 
 In the second case, the dynamics on $\{\varepsilon_3=0\}$ is defined by
 \begin{equation}
     \label{eq:invariantplane_second}
     \begin{split}
     \eta_3'=& \frac{\eta_3(\sin\eta_3^6+\cos\eta_3^6)H(\eta_3^6)}{6}\left[ 
     -ar_3^2+(\cos\eta_3^6-\sin\eta_3^6)H(\eta_3^6)
     \right]\\
     r_3'= &  \frac{r_3(\sin\eta_3^6+\cos\eta_3^6)H(\eta_3^6)}{2}\left[
     ar_3^2-(\cos\eta_3^6-\sin\eta_3^6)H(\eta_3^6)\right]\\
    & -\eta_3^6r_3(\cos\eta_3^6-\sin\eta_3^6)H(\eta_3^6)\left[  
     ar_3^2- (\cos\eta_3^6-\sin\eta_3^6)H(\eta_3^6)
     \right].
     \end{split}
 \end{equation}
A line of equilibria is thus given by the graph of the function
\begin{equation}
\label{eq:equi_invplane2}
    r_3=\sqrt{\frac{(\cos\eta_3^6-\sin\eta_3^6)H(\eta_3^6)}{a}}.
\end{equation}
The geometry of the orbits can be deduced by taking $r_3'/\eta_3'$, which yields
\[
    \frac{dr_3}{d\eta_3}=-3r_3\left( \frac{\sin\eta_3^6+\cos\eta_3^6 -\eta^5(\cos\eta_3^6-\sin\eta_3^6)}{\eta(\sin\eta_3^6+\cos\eta_3^6)}\right).
\]
The equation above is of separable variables and can be solved explicitly. The solution for the initial value problem $r_3(\eta_{3,0})=r_{3,0}$ is thus given by
\[
r_3(\eta_3)=R_0\cdot\frac{\sqrt{\sin\eta_3^6+\cos\eta_3^6}}{\eta_3^3}, \qquad R_0=\frac{r_{3,0}\eta_{3,0}^3}{\sqrt{\sin\eta_{3,0}^6+\cos\eta_{3,0}^6}}.
\]
Notice that there exists a unique point $F^*=(\eta^*_3,r_3^*)$ in which the orbits and the curve \eqref{eq:equi_invplane2} are tangent. For a global picture of the dynamics in $\{\varepsilon_3=0\}$ see Figure~\ref{fig:Exit_chart_planes} (b).

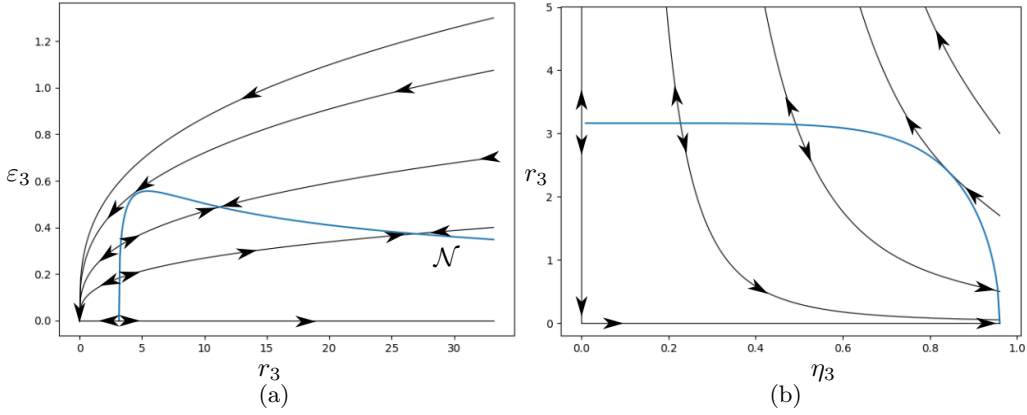
\begin{figure}[ht]
\centering
    \begin{subfigure}{0.45\textwidth}
     \centering
      \begin{overpic}[width=\linewidth]{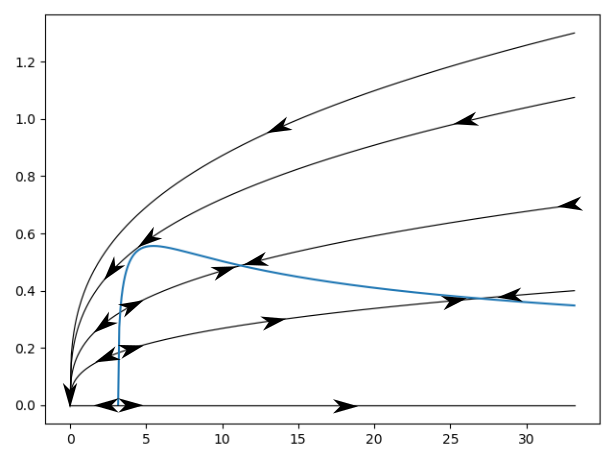}
      \end{overpic}
      \caption{}
    \end{subfigure}
        \begin{subfigure}{0.45\textwidth}
        \centering
    \begin{overpic}[width=\textwidth]{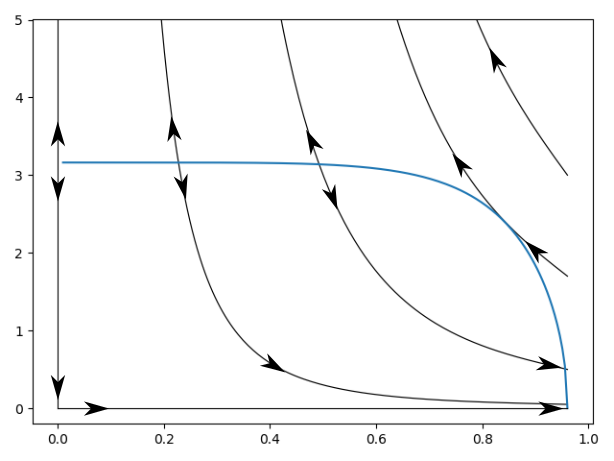}
    \put(-55,-2){$r_3$}
    \put(55,-2){$\eta_3$}
    \put(-105,37){$\varepsilon_3$}
    \put(-2,37){$r_3$}
    \put(-20,20){$\mathcal{N}$}
    \end{overpic}
        \caption{}
    \end{subfigure}

    \hfill
       \caption[Dynamics in exit chart]{The dynamics on the exit chart $K_3$, restricted to the planes $\{\eta_3=0\}$ and $\{\varepsilon_3=0\}$ in (a) and (b), respectively. In (a), orbits (in black, with arrows) are organised around the curve $\mathcal{N}$ (in blue); in (b), orbits are organised around the line of equilibria (in blue) given by \eqref{eq:equi_invplane2}.}
        \label{fig:Exit_chart_planes}
\end{figure}

We verify the existence of a centre manifold based on $E=\left(0,1/\sqrt{a},0\right)$. Indeed, the equilibrium point $E$ belongs to both lines of equilibria defined by \eqref{eq:equi_invplane1} and \eqref{eq:equi_invplane2}. In fact, its linear stability is given by the matrix
\[
    \left[ 
    \begin{array}{ccc}
    0 & 0 & 0 \\
    0 & 1 & 0 \\
    0 & 0 & 0
    \end{array}\right].
\]
Therefore, by the centre manifold theorem there exists a repelling 2-dimensional local centre manifold $W_{loc}^c(E)$ at the equilibrium $E$ which is tangent to the plane $\left\{r_3=\frac{1}{\sqrt{a}}\right\}$, and it is defined by a smooth function $r_c(\eta_3,\varepsilon_3)$. Moreover, $\{(0,r_c(0,\varepsilon_3),\varepsilon_3) : \varepsilon_3\in[0,\tilde{\gamma}_{\TU{out}}]\}$ coincides with 
the curve defined by \eqref{eq:equi_invplane1}, and analogously $\{(\eta_3,r_c(\eta_3,0),0) : \eta_3\in[0,\alpha_1]\}$ coincides with 
the curve defined by \eqref{eq:equi_invplane1}, for any $\alpha_1<\eta_3^*$.

Recall that the section $\Sigma_1$, c.f. \eqref{eq:transverseFirst}, is independent of $\epsilon$. We can lift it into chart $K_3$ as 
\begin{equation}
\label{eq:S1_K3}
  \tilde{\Sigma}_1=\left\{ (\tilde{\alpha}_1,r_3,\varepsilon_3) : r_3\in[0,\tilde{\beta}_1], \ \varepsilon_3\in[0,\tilde{\gamma}_1]\right\},  
\end{equation}
where $\tilde{\alpha}_1=\alpha_1^{1/6}$, $\tilde{\beta}_1=\beta_1\alpha_1^{1/2}$, and $\tilde{\gamma}_1$
is sufficiently small. We bring all together in the following proposition, where the transition map is constructed. 
\begin{proposition}
    \label{LEMMA:transition_exit}
        For each transversal $\tilde{\Sigma}_1$ (cf. \eqref{eq:S1_K3}) there is $\Sigma_{\TU{out}}$ as in \eqref{eq:Sigma_out_K3}, such that the transition map $\Pi_{\TU{out}}:\Sigma_{\TU{out}}\rightarrow\tilde{\Sigma}_1$ is well defined. Moreover, for each $\eta_3=\eta_{3,0}>0$ fixed, the set 
    \[\left\{ \Pi_{\TU{out}}\left(\eta_{3,0},r_3,\tilde{\gamma}_{\TU{out}}\right) : r_3\in[0,\tilde{\beta}_{\TU{out}}]\right\}
    \subset \left\{ \eta_3=\tilde{\alpha}_1, \ \varepsilon_3=\frac{\eta_{3,0}\tilde{\gamma}_{\TU{out}}}{\tilde{\alpha}_1} \right\}
        \]
        is a segment of length $\mathcal{O}\left(\eta_{3,0}^3\right)$. More precisely, for each $\left(\eta_3,r_3,\tilde{\gamma}_{\TU{out}}\right)\in \Sigma_{\TU{out}}$ we have that
        \[            \Pi_{\TU{out}}\left(\eta_3,r_3,\tilde{\gamma}_{\TU{out}}\right) = \left( \tilde{\alpha}_1, \tilde{r}_1(\eta_3,r_3),\frac{\eta_3\tilde{\gamma}_{\TU{out}}}{\tilde{\alpha}_1} \right),
        \]
   where $\tilde{r}_1=\Theta \left(r_3\eta_3^3\right)$.
\end{proposition}
\begin{proof}
    See Appendix~\ref{APPENDIX:proof_transition}, and Figure~\ref{fig:Exit_chart} therein. 
\end{proof}

The following result is a consequence of Propositions \ref{LEMMA:entrance_transition}, \ref{LEMMA: transit_fold_blowup}, and \ref{LEMMA:transition_exit}.
\begin{lemma}
    \label{LEMMA:transition4}
        The transition map $\Pi_{41}:\Sigma_4\rightarrow\Sigma_1$ is well defined for every $\epsilon$ sufficiently small. Moreover, the transition map $\Pi_{41}$ is a contraction with contraction rate $\mathcal{O}\left( e^{-c_4/\sqrt{\epsilon}} \right)$ for some $c_4>0$. More precisely, for each $(\omega,\beta_4)\in\Sigma_4$ we have that
        \[
            \Pi_{41}(\omega,\beta_4)=\left(
            \alpha_1,\hat{r}_1           \right),
        \]
        for each $\epsilon$ sufficiently small, where $\hat{r}_1=\Theta\left( \epsilon^{3/2} \right)$. Equivalently, in the original $\theta$ variable,
         \[
            \Pi_{41}(\theta,\beta_4)=\left(
            \pi/4+ \alpha_1, \hat{r}_1
            \right)
        \]
\end{lemma}
\begin{proof}
    We consider 
    \[
        \Pi_{41}=\Phi\circ\Pi_{\TU{out}}\circ T_{23}\circ \Pi_{\TU{tran}}\circ T_{12}
        \circ\Pi_{\TU{in}}\circ \Phi^{-1},
    \]
    which is well defined as a map from $\Sigma_4$ to $\Sigma_1$ for every $\epsilon$ sufficiently small. Since each $\Pi_{\TU{in}}, \Pi_{\TU{tran}}$, and $\Pi_{\TU{out}}$ are contractions with contraction rates $\mathcal{O}\left( e^{-\tilde{c}_1/\varepsilon_1} \right)$, $\mathcal{O}\left( e^{-\tilde{c}_2/\eta_2^6} \right)$, and $\mathcal{O}\left( \eta_3^3 \right)$, respectively, then the composition has a rate of contraction $\mathcal{O}\left( e^{-\tilde{c}_4/{\varepsilon_1}} \right)$ for some $\tilde{c}_4>0$. From \eqref{eq:coordinates_charts}, it follows that $\Pi_{41}$ is a $\mathcal{O}\left(  e^{-c_4/\epsilon^3}\right)$.
    
    In order to get the expression of $\Pi_{41}$, and for the reader's convenience, we recall from \eqref{eq:coordsChange_second} the change of charts $T_{12}$ and $T_{23}$ between charts $K_1$ and $K_2$, and between $K_2$ and $K_3$, respectively,
    \begin{align*}
        T_{12}(\omega_1,\eta_1,\varepsilon_1)=
         \left(  
        \frac{\omega_1}{\varepsilon^6_1},\frac{1}{\varepsilon_1^3},\eta_1\varepsilon_1
        \right), 
        \quad
        T_{23}(\omega_2,r_2,\eta_2)=
         \left(  
        \eta_2\omega_2^{1/6}, \frac{r_2}{\omega_2^{1/2}},\frac{1}{\omega_2^{1/6}}
        \right),
    \end{align*}
and the transition maps $\Pi_{\TU{in}},\Pi_{\TU{tran}},$ and $\Pi_{\TU{out}}$ from Propositions \ref{LEMMA:entrance_transition}, \ref{LEMMA: transit_fold_blowup}, and \ref{LEMMA:transition_exit},
\begin{align*}
    \Pi_{\TU{in}}(\omega_1,\tilde{\beta}_4,\varepsilon_1)= &
    \left( 
     \tilde{\omega}_1(\omega_1,\varepsilon_1)
    +\mathcal{O} \left( e^{-\tilde{c}_1/\varepsilon_1}\right), 
    \frac{\varepsilon_1\tilde{\beta_4}}{\gamma_{\TU{in}}},\gamma_{\TU{in}}
    \right), \\
    \Pi_{\TU{tran}}(\omega_2,\tilde{\beta}_{\TU{in}},\eta_2) = &
    \left(
    \omega_2^*+\alpha_{\TU{out}},r_2^*+\Theta(\eta_2^4),\eta_2
    \right), \\
    \Pi_{\TU{out}}(\eta_3,r_3,\tilde{\gamma}_{\TU{out}})= &
    \left( 
    \tilde{\alpha}_1,\tilde{r}_1(\eta_3,r_3), \frac{\eta_3 \tilde{\gamma}_{\TU{out}}}{\tilde{\alpha}_1}
    \right),
\end{align*}
where $\tilde{r}_1(\eta_3,r_3)=\Theta\left(r_3\eta_3^3 \right)$. Therefore, for any $(\omega_1,\tilde{\beta}_4,\varepsilon_1)\in\tilde{\Sigma}_1$,
\[
    (T_{12}\circ\Pi_{\TU{in}})(\omega_1,\tilde{\beta}_4,\varepsilon_1)=
    \left(
    \frac{\tilde{\omega}_1}{\gamma_{in}^6}+\mathcal{O} \left(e^{-\tilde{c}_1/\varepsilon_1} \right), \frac{1}{\gamma_{\TU{in}}^3},\tilde{\beta}_4\varepsilon_1
    \right).
\]
It follows that
\[
(\Pi_{\TU{tran}}\circ T_{12}\circ\Pi_{\TU{in}})(\omega_1,\tilde{\beta}_4,\varepsilon_1)= \left(\omega_2^*+\alpha_{\TU{out}},r_2^*+\Theta(\varepsilon_1^4),\tilde{\beta}_4\varepsilon_1 \right).
\]
Hence,
\[
(T_{23}\circ \Pi_{\TU{tran}}\circ T_{12}\circ\Pi_{\TU{in}})(\omega_1,\tilde{\beta}_4,\varepsilon_1)=
\left( \tilde{\beta}_4\varepsilon_1(\omega_2^*+\alpha_{\TU{out}})^{1/6}, \frac{r_2^*+\Theta(\varepsilon_1^4)}{(\omega_2^*+\alpha_{\TU{out}})^{1/2}}, \frac{1}{(\omega_2^*+\alpha_{\TU{out}})^{1/6}} \right).
\]
We conclude that
\[
(\Pi_{\TU{out}}\circ T_{23}\circ \Pi_{\TU{tran}}\circ T_{12}\circ\Pi_{\TU{in}})(\omega_1,\tilde{\beta}_4,\varepsilon_1)=
\left( 
\tilde{\alpha}_1, \tilde{r}_1(\omega_1,\varepsilon_1), \frac{\tilde{\beta}_4\varepsilon_1}{\tilde{\alpha}_1}\right)\in K_3,
\]
where 
\[
\tilde{r}_1(\omega_1,\varepsilon_1)\equiv \tilde{r}_1\left( 
\tilde{\beta}_4(\omega_2^*+\alpha_{\TU{out}})^{1/6}\varepsilon_1, \frac{r_2^*+\Theta\left( \varepsilon_1^4 \right)}{(\omega_2^*+\alpha_{\TU{out}})^{1/2}}
\right).
\]
Observe then that $\tilde{r}_1(\omega_1,\varepsilon_1)=\Theta\left( \varepsilon_1^3 \right)$. Recall that $\omega=\eta_3^6, r=\eta_3^3r_3$, and $\varepsilon=\eta_3\varepsilon_3$. Doing the substitution, and since $\varepsilon=\sqrt{\epsilon}$, we obtain
\[
    \Pi_{41}(\omega,\beta_4)=\left(\alpha_1, \hat{r}_1(\omega,\epsilon) \right)\in\Sigma_1,
\]
where $\hat{r}_1(\omega,\epsilon)=\Theta(\epsilon^{3/2})$, and the result follows.
\end{proof}

With the last transition map constructed we can finally prove Theorem~\ref{THM:PoincareMap}.
\begin{proof}[Proof of Theorem~\ref{THM:PoincareMap}]
Part (i) follows as a direct consequence of Lemmas~\ref{LEMMA:Transition1}, \ref{LEMMA:transition2}, \ref{LEMMA:transition3}, and \ref{LEMMA:transition4}. In order to get part (ii), recall that for every $\epsilon$ sufficiently small
\begin{align*}
    \Pi_{12}(z_1)= & 
    \left( \theta^*+\alpha_2,\phi_0(\theta^*+\alpha_2) \right) + \Theta(\epsilon) \\
    \Pi_{23}(z_2)= & 
    \left( \theta^*-\alpha_3,r_{\rho_*}(\theta^*-\alpha_3) \right) +\Theta(\epsilon^{2/3}) \\
    \Pi_{34}(z_3)= & 
    \left( \pi/4+\Theta(\epsilon^{3/2}),\beta_4)\right)\\
    \Pi_{41}(z_4)= & 
    \left( \frac{\pi}{4}+ \alpha_1,\hat{r}_1(z_4)\right), 
\end{align*}
for any $z_i\in\Sigma_i$, $i=1,2,3,4$, where in particular $\hat{r}_1(z_4)\equiv \hat{r}_1(\omega,\epsilon)$ as given in Lemma~\ref{LEMMA:transition4}. By performing the corresponding compositions, it yields that $\Pi^\epsilon(z_1)=(\pi/4+\alpha_1, \bar{r}_1(z_1))$, where $\bar{r}_1(z_1)= \Theta\left( \epsilon^{3/2} \right)$. The order of the contraction rate is inherited from the smallest of the $\Pi_{ij}$, being here $\mathcal{O}\left( e^{-c/\epsilon^3} \right)$, for some $c>0$.

Part (iii) follows directly from above since we have
\[
    z_{\epsilon}=\Pi^{\epsilon}(z_\epsilon)=z_0+\Theta\left(  \epsilon^{3/2}\right).
\]

Finally, part (iv) is a consequence of the following reasoning: given $z_{\epsilon}$ the fixed point of $\Pi^{\epsilon}$, there is $\rho_\epsilon\in(0,\rho^*)$ such that $z_{\epsilon}\in \mathcal{F}_{\epsilon}(\rho_{\epsilon})$. Therefore, it is attracted exponentially fast to $\mathcal{S}_{\epsilon}^1$. Once the trajectory follows $\mathcal{S}^2_{\epsilon}$, Remark~\ref{RMK:timescale1} indicates that the time spent near $\mathcal{S}^1_{\epsilon}$ is $\Theta(1/\epsilon)$. When the trajectory hits $\Sigma_3$, say in a point $\tilde{z}\in \Sigma_3$, there is $\tilde{\rho}_\epsilon>\rho_*$ such that $z\in\mathcal{F}_{\epsilon}(\tilde{\rho}_\epsilon)$, and thus is attracted exponentially fast to $\mathcal{S}^1_\epsilon$. Due to Remark~\ref{RMK:timescale2}, the time spent near $\mathcal{S}^1_{\epsilon}$ is $\Theta(\epsilon^{-3/2})$.
\end{proof}

\section{Proof of the main theorem}
\label{SEC:MainProof}
In this last section, we conclude with the proof of Theorem~\ref{THM:limit cycle_ original coordinates}. While the basic structure is a translation of Theorem~\ref{THM:PoincareMap} into the coordinates $x$ and $y$, the calculation of the time scales near each  branch given by the curves $\sigma_i$, $i=1,2,3,4$ requires some careful analysis.

\begin{proof}[Proof of Theorem~\ref{THM:limit cycle_ original coordinates}]
   Recall that the coordinates $(\theta,r)$ were obtained from the change of variables \eqref{eq: ChangeCoords_First} and \eqref{eq:coordsChange_second}, so that
    \begin{equation}
    \label{eq:coordfinal}
        x=\frac{\cos\theta}{\sqrt{\epsilon} r}, \qquad y=\frac{\sin\theta}{\sqrt{\epsilon} r}.
    \end{equation}
Let $z_{\epsilon}=\left(\bar{\alpha}_1,r_\epsilon\right)$, where $\bar{\alpha}_1=\pi/4+\alpha_1$, be the fixed point of $\Pi^{\epsilon}$ as above. Therefore, there exists a unique $\rho_\epsilon>0$ such that $z_{\epsilon}\in\mathcal{F}_0(\rho_{\epsilon})$, which implies that $r_{\epsilon}=\rho_{\epsilon}\sin\bar{\alpha}_1$, and due to Lemma~\ref{LEMMA:transition4} we have that $\rho_{\epsilon}=\Theta\left( \epsilon^{3/2}\right)$. Hence, by replacing $r=\rho_{\epsilon}\sin\theta$ in \eqref{eq:coordfinal}, the curve $\sigma_1$ is obtained.

The curves $\sigma_2$ and $\sigma_3$ are obtained similarly by replacing $r=\phi_0(\theta)$ (cf. \eqref{eq:criticalmanifolds}) and $r=\rho_*\sin\theta\equiv \frac{1}{2\sqrt{a}}\sin\theta$, respectively. For $\sigma_{4}$, we substitute $\theta=\pi/4$.

In order to obtain the time scales in each branch, recall from \eqref{eq:time_change} and \eqref{eq:rescale} that our time variable $t_2$ for the system \eqref{eq:MainBrusselator} is given by
\begin{equation}
\label{eq:t2_last}
    t_2(t)=\int_{0}^t{\frac{1}{r^2(s)}ds}.
\end{equation}
Whenever $r(t)$ is bounded away from $r=0$, the time scale analysis becomes simple. However, this is not the case near $\sigma_1$. Let us consider the following statement.

\noindent\textit{Claim: for all positive $\epsilon$ and $\rho$ sufficiently small, the vector field of system \eqref{eq:MainBrusselator} on the curves
$\{(\theta,r_{\rho}(\theta)): \theta\in[\alpha_1,\pi/2], \ \bar{\alpha}_1>\pi/4\}$ points towards their epigraphs $\{(\theta,r) : r\geq r_{\rho}(\theta)\}$.}

Indeed, for each $\rho>0$ a parametrisation of the fast fibre is given by $\theta \mapsto (\theta, r_{\rho}(\theta))$. A normal vector pointing in the direction of its epigraph at the point $(\theta,\rho\sin\theta)$ is given by $\vec{n}=(-\rho\cos\theta,1)$. Denote by $F\equiv F(\theta,r)$ the vector field defining \eqref{eq:MainBrusselator}, for which it is sufficient to show that
\[
    \left\langle F(\theta,\rho\sin\theta), \vec{n}\right\rangle >0.
\]
By a straightforward calculation, we have that on $(\theta,r_{\rho}(\theta))$
\[
     \left\langle F, \vec{n}\right\rangle = \epsilon \left[
     \rho^3\sin^2\theta (\sin\theta-\cos\theta) -\sqrt{\epsilon}a\rho^4\sin^3\theta
     \right],
\]
so that $\left\langle F, \vec{n}\right\rangle >0$ if and only if 
\[
    \sqrt{\epsilon}\rho < \frac{\sin\theta-\cos\theta}{a\sin\theta}
\]
for all $\theta\in[\bar{\alpha}_1,\pi/2]$. A sufficient condition for the claim to hold is to take $\epsilon$ and $\rho$ small enough so that $\sqrt{\epsilon}\rho < (\sin\bar{\alpha}_1-\cos\bar{\alpha}_1)/a$.

As a consequence of the proven claim, we have that the evolution of $r(s)$ as given by \eqref{eq:MainBrusselator}, with $r(0)=z\in\Sigma_1$, satisfies $r(s)\geq \rho \sin\theta(s)$ for all $s\geq 0$, see Figure~\ref{fig:finalproof}.

\begin{figure}[ht]
     \centering      \begin{overpic}[width=.5\linewidth]{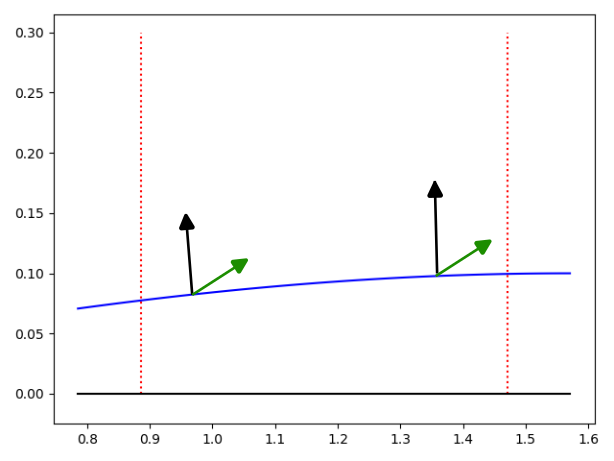}
     \put(20,7){\scriptsize$\alpha_1$}
     \put(75,7){\scriptsize$\pi/2-\delta$}
      \put(47,25){\color{blue}\scriptsize $\rho_{\epsilon}\sin\theta$}
        \put(70,47){\scriptsize $\vec{n}$}
        \put(29,41){\scriptsize $\vec{n}$}
        \put(37,34){\scriptsize $F$}
        \put(78,37){\scriptsize $F$}
     \end{overpic}  
     \hfill
    \caption[Lower bound for r]{A representation of the lower bound estimates for $r(s)$. In blue, the graph of $r_{\rho_{\epsilon}}(\theta)=\rho_\epsilon\sin\theta$ is depicted, on which the vector field $F$ (green) points towards the epigraph of $r_{\rho_{\epsilon}}$. As a reference,  normal vectors (black) to two different points on the curve are shown.}
        \label{fig:finalproof}
\end{figure}
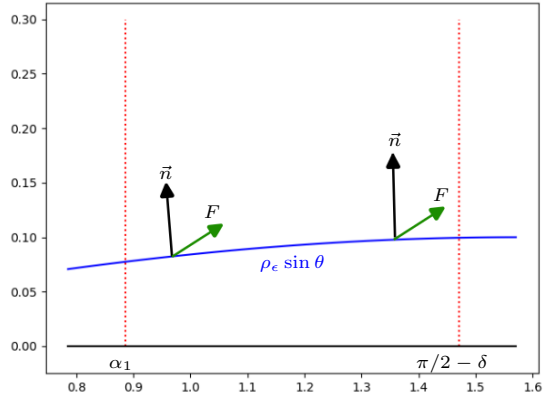

Recall from part (iv) in Theorem~\ref{THM:PoincareMap} that near $\hat{\sigma}_1$, i.e. $\{r=0\}$, the time scale is $t_2=\Theta(1)$, so that in particular $t_2\geq M_1$ for some $M_1>0$. From \eqref{eq:t2_last} and the claim above, it follows that there exists $M_2>0$ such that 
\[
    M_1\leq \int_0^t\frac{ds}{\rho_\epsilon^2\sin^2\theta(s)}\leq 
    \frac{t}{M_2\epsilon^{3}}.
\]

Therefore, the time scale near $\sigma_1$ is $\Omega\left( \epsilon^3 \right)$, for some $\tilde{\delta}>0$.

We now show that the time scale near $\sigma_1$ is $\mathcal{O}\left( \epsilon^{3} \right)$. For this, we first find an upper bound for the dynamics of $r(s)$. From the equation for $r$, bounding from above the first, second, and therm terms in \eqref{eq:MainBrusselator} we have that for $\theta\in [\pi/4+\delta,\pi/2-\delta]$ for $\delta$ arbitrarily small
\[
\dot{r}\leq r- (M-\epsilon)r^3,
\]
where $M=a\cos(\pi/2-\delta)[\sin(\pi/4+\delta)-\cos(\pi/4+\delta)]$ and $M\rightarrow 0$ as $\delta\rightarrow0$. For simplicity, let $F_{\epsilon}=M-\epsilon>0$. It follows that for $\delta$ sufficiently small, $1-F_{\epsilon}r^2(s)>0$ for all $s$, and thus
\[
\int_{r_0}^{r}{\frac{d\omega}{\omega(1-F_{\epsilon}\omega^2)}}\leq s.
\]
The integral on the right can be solved by partial fractions method, by using the expansion
\[
    \frac{1}{\omega(1-\sqrt{F_\epsilon}) (1+\sqrt{F_\epsilon})} =
    \frac{1}{\omega} + \frac{\sqrt{F_\epsilon}}{2(1-\sqrt{F_\epsilon}\omega)}- \frac{\sqrt{F_\epsilon}}{2(1+\sqrt{F_\epsilon}\omega)}.
\]
Therefore, 
\[
\ln\left(\frac{r}{r_0}\right)+ \frac{1}{2}\ln \left(\frac{1-F_\epsilon r_0^2}{1-F_\epsilon r^2}\right)\leq s,
\]
and by multiplying by $2$ and taking the exponential function we obtain
\[
    r^2(1-F_\epsilon r_0^2) \leq 
    (r_0^2-F_\epsilon r_0^2 r_2^2)e^{2s},
\]
from where it follows that 
\begin{equation}
\label{eq:bound_r}
r^2(s)\leq \frac{r_0^2 e^{2s}}{1+r_0^2 F_\epsilon (e^{2s}-1)}
\end{equation}

As above, we consider the initial condition to be the fixed point of the Poincaré map $\Pi^{\epsilon}$, so that $r_0=\rho_{\epsilon}\sin\bar{\alpha}_1$. By using \eqref{eq:bound_r} in \eqref{eq:t2_last},
\[
    t_2\geq \frac{1}{r_0^2}\int_{0}^t{\left(\frac{1-r_0^2 F_\epsilon}{e^{2s}}+r_0^2F_\epsilon\right)ds}\geq
    \frac{(1-r_0^2 F_\epsilon)(1-e^{-2t})}{2r_0^2}.
\]
Since $t_2=\mathcal{O}(1)$ near $\{r=0\}$, there is $M>0$ such that
\[
    1-e^{-2t}\leq \frac{2Mr_0^2}{1-r_0^2 F_\epsilon}.
\]
Hence,
\[
    -2t \geq \ln\left( 1-\frac{2Mr_0^2}{1-r_0^2F_\epsilon} \right)\geq - \frac{2Mr_0^2}{1-r_0^2(F_\epsilon+2M)}.
\]
Therefore,
\[
    t\leq \frac{M}{1-r_0^2(F_\epsilon+2M)} r_0^2.
\]
Since $r_0^2=\mathcal{O}\left( \epsilon^{3}\right)$, the same holds for the time scale near $\sigma_1$.

Last, for $\sigma_2, \sigma_3$, and $\sigma_4$ we have that $r$ is bounded away from $0$ and bounded from above, so that from part (iv) in Theorem~\ref{THM:PoincareMap} and \eqref{eq:t2_last} it follows that
\begin{align*}
    \textup{near } \sigma_2, \quad & \quad t= \Theta (\epsilon^{-1}) \\
    \textup{near } \sigma_3, \quad& \quad t= \Theta(1) \\
   \textup{near } \sigma_4, \quad & \quad t= \Theta\left( \epsilon^{-3/2}\right).
\end{align*}
This finishes the proof.
\end{proof}

As mentioned at the end of Section~\ref{SUBSEC:time_scale_sep}, Theorem~\ref{THM:limit cycle_ original coordinates} can be reformulated in terms of the rescaled variables $(\bar{x},\bar{y})=(\sqrt{\epsilon}x,\sqrt{\epsilon}y)$ as given in the following corollary, whose proof is a direct consequence of Theorems~\ref{THM:limit cycle_ original coordinates} and \ref{THM:PoincareMap}.
\begin{corollary}
\label{CORO:rescaled}
    For each $\epsilon>0$ sufficiently small, system \eqref{eq:MainBrusselator}, after performing the rescaling $(\bar{x},\bar{y})=\sqrt{\epsilon}\;(x,y)$, admits a unique attracting limit cycle $\bar{\gamma}^\epsilon$ which exhibits a time scale separation near the cycle composed of four curves parameterised by the functions $\bar{\sigma}_i=\sqrt{\epsilon}\sigma_i$, $i=1,2,3,4$, where each $\sigma_i$ is as in Theorem~\ref{THM:limit cycle_ original coordinates}. The time scale near each $\bar{\sigma}_i$ corresponds to that near $\sigma_i$. Furthermore, 
    \begin{equation}
    \label{eq:limit_cycle}
        \lim_{\epsilon\rightarrow 0} \TU{dist}_H\left( 
        (\bar{\sigma}_2\cup \bar{\sigma}_3 \cup \bar{\sigma}_4)
        \cap \left\{\bar{y}\leq \frac{1}{\sqrt{2}\rho_\epsilon}\right\}, \bar{\gamma}^{\epsilon}
        \right)=0,
    \end{equation}
where $\TU{dist}_H$ denotes the Hausdorff semidistance between sets. 
\end{corollary}

\section{Conclusion and outlook}
\label{SEC:conclusion}

Employing techniques from geometric singular perturbation theory, we have established the existence of relaxation oscillations in the Brusselator when the bifurcation parameter $b=a/\epsilon$ takes on large values. We have shown that the system exhibits a four-stroke oscillator pattern, transitioning from a superslow via an ultrafast to a slow and then to a fast regime just to finish in a superslow regime, before repeating the cycle. One of the main achievements of this work is the determination of the precise time scales within the original time variable $t$, which is challenging due to the state-dependent time rescalings such as \eqref{eq:time_change}. In our case, however, it has been possible to determine in particular the ultrafast time scale by investivating the behaviour of  the fixed point $z_{\epsilon}$ of the constructed Poincaré map $\Pi^{\epsilon}$ as $\epsilon\rightarrow 0$.

As stated in Theorem~\ref{THM:limit cycle_ original coordinates}, the superslow regime represents the slowest among the four distinguished time scales. Conversely, the ultrafast regime manifests as dynamics which, in the limit as $\epsilon\rightarrow0$, resemble an instantaneous jump of infinite length (see also the time series in Figure~\ref{fig:limitcycle_growing} (b)). As already suggested in \cite{EngelOlicon23}, we conjecture that, under the presence of a noisy perturbation as formulated in \cite{ArnoldBrus}, it is this transition between the superslow and ultrafast regimes  
which causes \textit{noise-induced finite-time instabilities} in the stochastic version of the Brusselator. We anticipate that extensions of the Brusselator, for instance where spatiotemporal dynamics or nonautonomous perturbations are involved, exhibit interesting phenomena based on the underlying dynamical features described in this paper.

These motivating examples may require a deeper understanding of the time scales involved around the nonhyperbolic equilibrium $P_0$, the partially hyperbolic equilibrium $P_1$, the fold point $F$, or the drop point $Q$, which have not been addressed in the present work.  

While the Brusselator serves as a foundational model for understanding chemical oscillations, one should bear in mind its purely theoretical character. Nevertheless, the GSPT analysis presented in this work, akin to previous studies of chemical oscillators (for instance \cite{Szmolyan09,Kosiuk11}), indicates that GSPT may be a powerful tool applicable to more realistic systems. As mentioned in \cite{Nicolis71}, time scale separations are ubiquitous in biochemical reaction networks due to the diverse mechanisms which may be involved within the reaction system. Recent studies of an urea-urease reaction as an exemplary system of a pH oscillator \cite{Winkelmann23, Straubeetal} seem to exhibit complicated fast-slow structures, induced by large differences in the values of the parameters. Such pH oscillators may well be understood by means of GSPT techniques, as elaborated in this work.

\subsection*{Acknowledgements} 
We acknowledge the support of Deutsche Forschungsgemeinschaft (DFG) through CRC 1114 and under Germany's Excellence Strategy -- The Berlin Mathematics Research Center MATH+ (EXC-2046/1, project 390685689).  
M.~E. additionally thanks the DFG-funded SPP 2298 and the Einstein Foundation for supporting his research.
G.~O.-M. also thanks FU Berlin for a 3-month Forschungsstipendium. The authors gratefully acknowledge Peter Szmolyan, Hildeberto Jard\'on-Kojakhmetov, and Samuel Jelbart for fruitful and insightful discussions. 

\appendix

\section{Taylor expansion of the slow manifolds}
In this Appendix we provide the first order approximation of both slow manifolds $\mathcal{S}_\epsilon^{1,2}$, given by Propositions \ref{PROP:expansionSlowManifold} and \ref{PROP:S1_approximation}.

\subsection{Approximation of \texorpdfstring{$\mathcal{S}^2_{\epsilon}$}{S2}}
\label{SEC:expansions}

\begin{proof}[Proof of Proposition~\ref{PROP:expansionSlowManifold}] 
Consider a general smooth 2-dimensional ODE of the form 
\begin{equation}
    \label{eq:chart1_expanded}
    \begin{array}{rcl}
        \theta'&=& f_0(\theta,r)+\varepsilon f_1(\theta,r)+\varepsilon^2 f_2(\theta, r)+ \mathcal{O}\left(\varepsilon^3\right) \\
        r' &=& g_0(\theta,r)+\varepsilon g_1(\theta,r) + \varepsilon^2 g_2(\theta,r) +\mathcal{O}\left(\varepsilon^3\right),
    \end{array}
\end{equation}
admitting a compact normally hyperbolic invariant manifold when $\varepsilon=0$, given by the graph of a smooth function $\phi_0(\theta)$. Therefore, for each $\varepsilon>0$ sufficiently small, the invariant manifold persists and is given by the graph of a smooth function $\phi_\varepsilon$. We thus consider the ansatz
\begin{equation}
\label{eq:ansatzSlowManifold}
    \phi_{\varepsilon}(\theta)=\phi_0(\theta)+\varepsilon\phi_1(\theta)+\varepsilon^2\phi_2(\theta)+\mathcal{O}\left(\varepsilon^3\right),
\end{equation}
which we plug in \eqref{eq:chart1_expanded}. Notice also that for $h\in\{f_1,f_2,g_1,g_2\}$ its Taylor expansion up to $\mathcal{O}(\varepsilon^2)$ terms is given by
\begin{equation}
    \label{eq:TaylorTerms}
        h(\theta,\phi_\varepsilon)= h(\theta,\phi_0)+
        \varepsilon \left( \partial_r h(\theta,\phi_0)\cdot \phi_1 \right)+
         \varepsilon^2\left( \frac{\partial^2_{rr}h(\theta,\phi_0) \cdot \phi_1^2}{2}+\partial_rh(\theta,\phi_0)\cdot \phi_2 \right) \nonumber 
        +\mathcal{O}\left(\varepsilon^3\right).
\end{equation}

Since $\phi_0$ is such that $f_0(\theta,\phi_0(\theta))=g_0(\theta,\phi_0(\theta))=0$, by comparing equal powers of $\varepsilon$ we obtain that
\begin{equation}
\label{eq:phi1}
    \phi_1= \frac{\phi_0'\cdot f_1-g_1}{\partial_r g_0-\phi_0'\cdot \partial_r f_0},
\end{equation}
\begin{equation}
\label{eq:phi2}
    \phi_2=\frac{\frac{1}{2}\phi_1^2(\phi_0'\cdot\partial_{rr}^2f_0-\partial_{rr}^2g_0) + \phi_1(\phi_0'\partial_r f_1+\phi_1'\partial_r f_0-\partial_r g_1)+f_1\phi_1'+f_2\phi_0'-g_2}{\partial_rg_0-\phi_0'\cdot\partial_rf_0}
\end{equation}
where all the expressions $f_i,g_i$ and their partial derivatives are evaluated in $(\theta,\phi_0)$.

In particular for \eqref{eq:MainBrusselator}, by doing the change of variable $\epsilon=\varepsilon^2$ and considering that $p(\theta,R)=aR^2-\cos\theta(\sin\theta-\cos\theta)$, we have that
\begin{equation}
\label{eq:expansionterms}
\begin{array}{rcl}
    f_0(\theta,r)&=& -\sin\theta(\sin\theta-\cos\theta)\cdot p(\theta,r)\\
    
    g_0(\theta,r)&=& -r\cos\theta(\sin\theta-\cos\theta)\cdot p(\theta,r),\\

f_1(\theta,r)&=& g_1(\theta,r) \;=\; 0 \\

f_2(\theta,r)&=& -r^2\cos\theta(\sin\theta-\cos\theta),\\

g_2(\theta,r)&=&r^3\sin\theta(\sin\theta-\cos\theta).

\end{array}
\end{equation}
On the one hand, from \eqref{eq:phi1} it follows that $\phi_1\equiv 0$. On the other hand, from \eqref{eq:phi2} and the fact that
\[
    \phi'(\theta)=\frac{-\sin\theta(\sin\theta-\cos\theta)+\cos\theta(\sin\theta+\cos\theta)}{2a\phi_0},
\]
we get 
\[
    \phi_2(\theta)=\frac{\phi_0 \left[ 
    \sin^2\theta\cos\theta-2\sin\theta\cos^2\theta-\cos^3\theta-2a\phi^2_0\sin\theta\right]}{2a\left[ -\sin^3\theta+2\sin^2\theta\cos\theta+\sin\theta\cos^2\theta-2a\phi_0^2\cos\theta\right]}.
\]
Substituting $a\phi_0^2=\cos\theta(\sin\theta-\cos\theta)$ yields
\[
    \phi_2(\theta)=\frac{-\phi_0(\theta)\cos\theta}{2a\left[ -\sin^3\theta+2\sin^2\theta\cos\theta-\sin\theta\cos^2\theta+2\cos^3\theta \right]}.
\]
Upon substituting $\sin^3\theta=\sin\theta-\sin\theta\cos^2\theta$ and $\sin^2\theta=1-\cos^2\theta$ in the denominator, we obtain
\[
    \phi_2(\theta)=-\frac{\phi_0(\theta)\cos\theta}{2a(2\cos\theta-\sin\theta)},
\]
and the result follows recalling $\epsilon=\varepsilon^2$.
\end{proof}

As a direct consequence, we can conclude  that the time scale spent near $\mathcal{S}^2_{\epsilon}$ is of order $\Theta\left( \varepsilon^{-2}\right)=\Theta \left( \epsilon^{-1} \right)$, as given in Remark~\ref{RMK:timescale1}. Indeed, upon substituting $r=\phi_{\epsilon}(\theta)$ in \eqref{eq:chart1_expanded}, due to \eqref{eq:TaylorTerms} and \eqref{eq:expansionterms}, we have that 
\[
    \theta'=\varepsilon^2 \left[ \partial_r f_0(\theta,\phi_0(\theta))\phi_2(\theta)+f_2(\theta,\phi_0(\theta)) \right] + o(\varepsilon^2).
\]
Since the right hand side is bounded from below away from $0$ and bounded from above, we have that there exist $\delta_1,\delta_2>0$ such that for any $t>0$
\[
    \delta_1 \varepsilon^2t \leq 
    \theta(t)-\theta(0) \leq
    \delta_2\varepsilon^2 t.
\]
This implies, in particular, that any suitable hitting time is of order $\Theta\left(\varepsilon^{-2}\right)= \Theta\left(\epsilon^{-1}\right)$.

\subsection{Approximation of \texorpdfstring{$\mathcal{S}^1_{\epsilon}$}{S1}}
\label{APPENDIX:approx_S1}
We now proceed to give the first order approximation of $S_{\epsilon}^1$. For simplicity, let us first consider $\theta=\pi/4+\omega$, where $\omega\in[0,\pi/4]$, and $\epsilon=\varepsilon^2$ as above, for which we refer to any slow manifold as $\mathcal{S}^1_{\varepsilon}\equiv \mathcal{S}^1_{\epsilon^{1/2}}$. Using the fact that
\[
    \sin\theta=\frac{\sin\omega +\cos\omega}{\sqrt{2}}, \qquad
    \cos\theta=\frac{ \cos\omega-\sin\omega}{\sqrt{2}},
\]
system \eqref{eq:MainBrusselator} expressed in the $\omega$-variable reads as 
\begin{equation}
\label{eq:wr_system_APP}
\begin{array}{rcl}
\omega'&=& -ar^2\sin\omega (\sin\omega+\cos\omega)+
\sin^2\omega (\sin\omega+\cos\omega)(\cos\omega-\sin\omega)\\
& & \qquad +\varepsilon^2\left[
-r^2\sin\omega(\cos\omega-\sin\omega) +
\varepsilon(a/\sqrt{2}) r^3(\cos\omega-\sin\omega)
\right],\\
\\
r'&=& -ar^3\sin\omega (\cos\omega-\sin\omega) +r\sin^2\omega (\cos\omega-\sin\omega)^2 \\
& &\qquad +\varepsilon^2\left[
r^3\sin\omega (\sin\omega +\cos\omega) -
\varepsilon(a/\sqrt{2})r^4 (\sin\omega +\cos\omega)
\right].
\end{array}
\end{equation}

\begin{proof}[Proof of Proposition~\ref{PROP:S1_approximation}]
    Recall that $S_0^1=\{\omega=0\}$, so that $S_\varepsilon^1$ is given as the graph of a function
\[
    \omega_\varepsilon(r)=\varepsilon \omega_1(r) + \varepsilon^2\omega_2(r) +
    \varepsilon^3\omega_3(r) +\mathcal{O}(\varepsilon^4).
\]
We write the system \eqref{eq:wr_system_APP} in an analogous form as in \eqref{eq:chart1_expanded}, so that it reads
\begin{align*}
    r' = & \ f_0(r,\omega)+ 
    \varepsilon f_1(r,\omega) + 
    \varepsilon^2 f_2(r,\omega)+ 
    \varepsilon^3 f_3(r,\omega)+ \mathcal{O}(\varepsilon^4)\\
    \omega' = & \ g_0(r,\omega)+ 
    \varepsilon g_1(r,\omega) + 
    \varepsilon^2 g_2(r,\omega)+ 
    \varepsilon^3 g_3(r,\omega)+ \mathcal{O}(\varepsilon^4),
\end{align*}
where,
\begin{align*}
    f_0(r,\omega)=& -ar^3\sin\omega (\cos\omega-\sin\omega) +r\sin^2\omega (\cos\omega-\sin\omega)^2 \\
    g_0(r,\omega)=& -ar^2\sin\omega(\sin\omega+\cos\omega) +\sin^2\omega (\sin\omega+\cos\omega)(\cos\omega-\sin\omega)\\
    f_1(r,\omega)= & \ g_1(r,\omega) = 0 \\
    f_2(r,\omega)= & r^3\sin\omega(\sin\omega+\cos\omega)\\
    g_2(r,\omega)= & -r^2\sin\omega(\cos\omega-\sin\omega) \\
    f_3(r,\omega)= & -\frac{ar^4}{\sqrt{2}} (\sin\omega+\cos\omega)\\
    g_3(r,\omega)= & \frac{a}{\sqrt{2}}r^3(\cos\omega-\sin\omega).
\end{align*}

Since $\omega_0\equiv 0$, it follows from \eqref{eq:phi1} that $\omega_1\equiv 0$. Therefore, by using \eqref{eq:phi2}, we have that $\omega_2\equiv 0$ as well. As a consequence, we obtain
\[
    \omega_3(r)= \frac{\omega_0'(r)\cdot f_3(r,\omega_0)-g_3(r,\omega_0)}{\partial_\omega g_0(r,\omega_0)-\omega_0'\cdot \partial_\omega f_0(r,\omega_0)} = \frac{r}{\sqrt{2}}.
\]
The result follows by replacing $\omega=\theta+\pi/4$ and $\varepsilon=\epsilon^{1/2}$.
\end{proof}

\section{Proof of exit transition}
\label{APPENDIX:proof_transition}
We devote this Appendix to prove Proposition~\ref{LEMMA:transition_exit}. Recall that we consider the sections (for appropriate positive constants $\tilde{\alpha}_{out},\tilde{\beta}_{out},\tilde{\gamma}_{out}, \tilde{\alpha}_1, \tilde{\beta}_1$, and $\tilde{\gamma}_1$)
\begin{align*}
    \Sigma_{out}=&\left\{
    (\eta_3,r_3,\tilde{\gamma}_{out}) : \eta_3\in[0,\tilde{\alpha}_{out}], r_3\in[0,\tilde{\beta}_{out}]
    \right\}, \\
    \tilde{\Sigma}_1=& \left\{
    \left(\tilde{\alpha}_1,r_3,\varepsilon_3\right) : r_3\in[0,\tilde{\beta}_1], \varepsilon_3\in[0,\tilde{\gamma}_1]
    \right\}.
\end{align*}

\begin{proof}[Proof of Proposition~\ref{LEMMA:transition_exit}]
We split the proof in six parts. First, we construct a compact region $\Xi$ for the system, where the section $\Sigma_{out}$ is contained in the boundary from where the vector field points to its interior, see Figure~\ref{fig:Exit_chart}. Secondly, we show that for any initial condition $z_0=(\eta_{3,0},r_{3,0},\varepsilon_{3,0})$ in $\Sigma_{out}$ with $\eta_{3,0}>0$ the solution for $r_3$ is bounded above and below by exponentially decaying functions. Afterwards, we prove that for such initial conditions the transition map is well defined, and in a subsequent step we give some logarithmic bounds on $T_+$. In the next step, we prove that the transition map $\Pi_{\TU{out}}$ is of the form 
\[
\Pi_{\TU{out}}\left(z_0\right)=\left(\tilde{\alpha}_1,\tilde{r}_3, \frac{\eta_{3,0}\gamma_{\TU{out}}}{\tilde{\alpha}_1} \right),
\]
where $\tilde{r}_3=\Theta\left(r_{3,0}\cdot \eta_{3,0}^3\right)$. Finally, we show that $\Pi_{out}$ is an exponential contraction for each $\eta_{3,0}>0$ fixed.

\noindent\textit{Step 1: Construction of $\Xi$.} Given $\tilde{\alpha}_1$ sufficiently small, take $\tilde{\alpha}_{out}<\tilde{\alpha}_1$. Let us consider the planar region 
\[\mathcal{R}=\left\{(\eta_3,\varepsilon_3) : \eta_3\varepsilon_3= k_0, \ \eta_3\in[0,\alpha_1], \ \varepsilon_3\in[0,\gamma_{out}], \  k_0\in[0,\alpha_1\gamma_{out}] \right\},\]
and consider the compact region (see Figure~\ref{fig:Exit_chart})
\begin{equation}
    \label{eq:Omega_entryexit}
    \Xi=\left\{ 
    (\eta_3,r_3,\varepsilon_3) : (\eta_3,\varepsilon_3)\in\mathcal{R}, \ r_3\in [0,\beta_{out}]
    \right\}.
\end{equation}

\begin{figure}[ht]
     \centering
      \begin{overpic}[width=.6\linewidth]{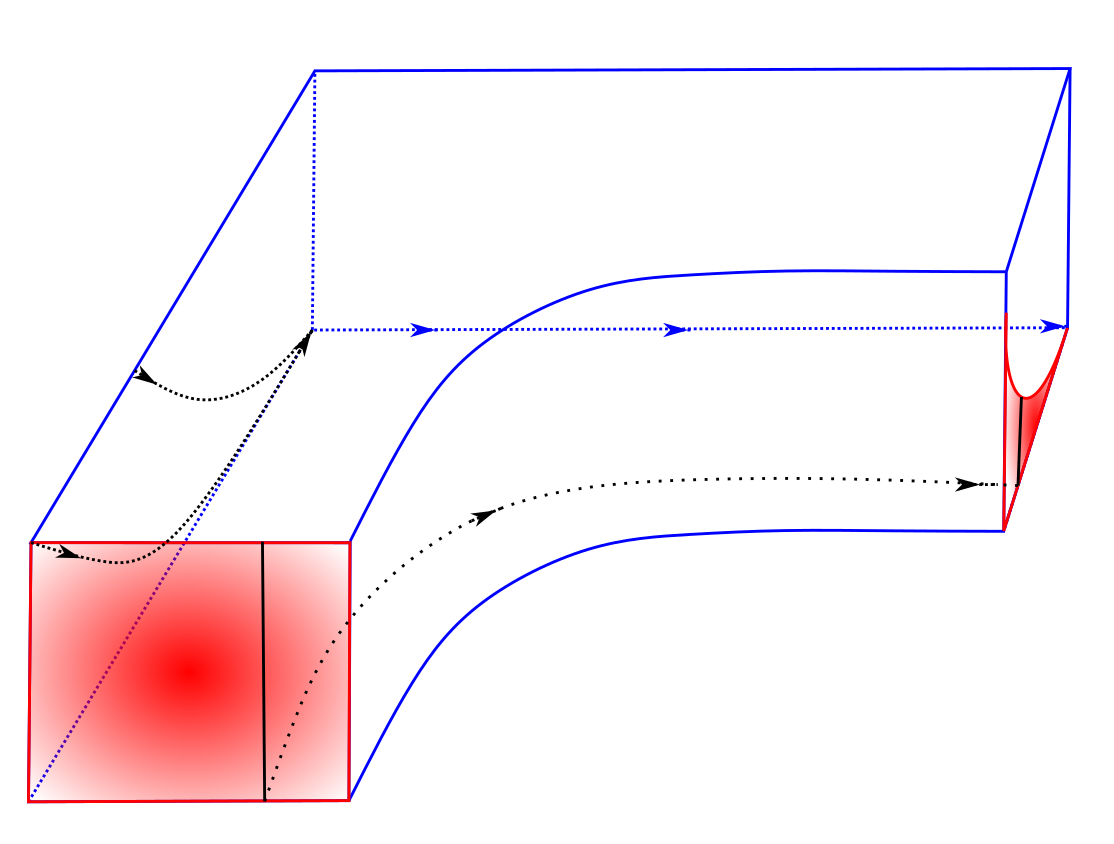}
      \put(12,55){$\Xi$}
      \put(6,20){$\Sigma_{out}$}
      \put(98,55){$\tilde{\Sigma}_1$}
      \end{overpic}
       \hfill
   \caption[Dynamics in exit chart]{A sketch of the transition $\Pi_{out}$ from $\Sigma_{out}$ to $\tilde{\Sigma}_1$ in chart $K_3$. Both $\Sigma_{out}$ and its image $\Pi_{out}(\Sigma_{out})\subset \tilde{\Sigma}_1$ are portrayed in red. Vertical lines in $\Sigma_{out}$ are mapped to vertical lines in $\tilde{\Sigma}_1$.}
        \label{fig:Exit_chart}
\end{figure}

In essence, the region $\Xi$ is bounded by the sections $\Sigma_{out}$ and $\tilde{\Sigma}_1$, and the planes $\{\eta_3=0\}$, $\{r_3=0\}$, $\{\varepsilon_3=0\}$, and $\{r_3=\beta_{out}\}$, so that orbits starting in $\Xi$ may escape only through $\Sigma_{out}$, $\{r_3=\beta_{out}\}$, or $\tilde{\Sigma}_1$. Let us rule out the first two possibilities by calculating $\varepsilon_3'$ and $r_3'$, respectively. Indeed, for simplicity let us denote $\tilde{\alpha}_1=\delta$, where $\delta\ll 1$, such that for all $\eta_3\leq \tilde{\alpha}_1$ we have
\begin{equation}
\label{eq:dummy3}
\begin{split}
    1 \leq &\sin\eta_3^6+\cos\eta_3^6 \leq 1+\delta^6,\\
    1-2\delta^6\leq &\cos\eta_3^6-\sin\eta_3^6 \leq 1,\\
    1-2\delta^6\leq &H(\eta_3^6) \leq 1.
\end{split}
\end{equation}
    Therefore, by substituting $\varepsilon_3=\gamma_{out}$ and $\eta_3=k_0/\gamma_{out}$ in \eqref{eq:blowchart2} we obtain
\begin{align*}
    \frac{6\varepsilon_3'}{\gamma_{out}}& \leq ar_3^2(1+\delta^6)-(1-2\delta^6)^3-\frac{a}{\sqrt{2}}+k_0^2r_3^2(1-2\delta^6) \\
    & \leq a\beta_{out}^2-1 +\mathcal{O}(\delta^6)+\mathcal{O}(k_0^2).
\end{align*}
Assuming that $\beta_{out}< \frac{1}{\sqrt{a}}$ we have that $\epsilon_3'<0$ on $\Sigma_{out}$, and thus the vector field defining \eqref{eq:blowchart2} points to the interior of $\Omega$ on $\Sigma_{out}$.

In a similar way, the vector field on $\{r_3=\beta_{out}\}$ points inwards, since by gathering the linear and cubic terms in $r_3$ in \eqref{eq:blowchart2} we get
\begin{equation}
\label{eq:dummy2}
\begin{split}
     r_3'\leq & r_3^3H(\eta_3^6)\left[ \frac{a+k_0^2}{2}(\sin\eta_3^6+\cos\eta_3^6)-a\eta_3^6(\cos\eta_3^6-\sin\eta_3^6) +\eta_3^6k_0^2(\sin\eta_3^6+\cos\eta_3^6) \right] \\
     & \quad -r_3 H^2(\eta_3^6) (\cos\eta_3^6-\sin\eta_3^6)\left[ \frac{\sin\eta_3^6+\cos\eta_3^6}{2} -\eta_3^6(\cos\eta_3^6-\sin\eta_3^6) \right],
\end{split}
\end{equation}
where the inequality holds since all the $\mathcal{O}(\eta_3^4)$ terms are negative. Using \eqref{eq:dummy3}, we obtain
\begin{equation}
\label{eq:dummy1}
    r_3'\leq -Cr_3+Dr_3^3,
\end{equation}
where
\begin{align*}
    C\equiv &C(\delta):= (1-2\delta^6)^3\left( \frac{1}{2}-\delta^6\right)= \frac{1}{2}+\Theta(\delta^6) \\
    D\equiv &D(\delta,k_0) :=  (1+\delta^6)\left[\frac{a+k_0^2}{2}+\delta^6 k_0^2\right]= \frac{a}{2} + \Theta\left( k_0^2+\delta^6 \right).
\end{align*}
Therefore, in the plane $\{r_3=\beta_{out}\}$ we have that
\[
    \frac{2r_3'}{\beta_{out}}\leq a\beta_{out}^2-1 + \mathcal{O}(k_0^2+\delta^6),
\]
and the claim follows since $\beta_{out}<1/\sqrt{a}$.

\noindent\textit{Step 2: exponential decay of $r_3$.} We now show that $r_3(t)$ is bounded above and below by an exponentially decreasing function as long as the orbit remains in $\Xi$. For each initial condition $(\eta_{3,0},r_{3,0},\varepsilon_{3,0})\in\Xi$, let $T_+$ be the first time its orbit leaves the region $\Xi$. From step 1 above, $T_+$ is simply
\[
    T_+\equiv T_+(\eta_{3,0},r_{3,0},\varepsilon_{3,0}):=
    \inf \{t\geq 0 : \eta_3(t)> \alpha_1\},
\]
that is the first time a trajectory starting at $(\eta_{3,0},r_{3,0},\varepsilon_{3,0})$ hits the section $\tilde{\Sigma}_1$. 

From \eqref{eq:dummy1}, it follows that 
\[
    r_3'\leq r_3\left( Dr_3^2-C \right),
\]
where the factor $Dr_3^2-C$ is negative if and only if $r_3\leq \sqrt{C/D}$. Since 
\[
    \lim_{\delta\rightarrow0}\lim_{k_0\rightarrow0} \frac{C}{D}=\frac{1}{a},
\]
there exists $k_0(\delta)$ sufficiently small such that for all $t\in[0,T_+]$
\[
    \frac{r_3'}{r_3\left(Dr_3^2-C\right)}\geq 1.
\]
Integrating from $0$ to $t$, factoring $1/D$ on the left side, and letting $F=\sqrt{C/D}$, one obtains
\[
    t\leq \frac{1}{D}\int_{r_3^0}^{r_3(t)}{\frac{ds}{s(s^2-F^2)}}=
    \frac{1}{C} \int_{r_3^0}^{r_3(t)}{\left[ -\frac{1}{s}+ \frac{1}{2(s-a)}+ \frac{1}{2(s+a)}\right]ds}.
\]
Hence,
\[
e^{2Ct}\leq \frac{r_{3,0}^2 (r_3-F)(r_3+F)}{r_3^2(r_{3,0}-F)(r_{3,0}+F)}.
\]
Since $r_{3,0}-F$ is negative, we conclude that
\[
    r_3^2\leq \frac{r_{3,0}^2 F^2}{1+ (F^2-r_{3,0}^2)e^{2Ct}}\leq K^2r_{3,0}^2e^{-2Ct},
\]
implying that $r_3(t)\leq K r_{3,0}e^{-Ct}$, where $K:=F^2/(F^2-\tilde{\beta}_{out}^2)$. In other words, if $t\leq T_+$, we have $r_3(t)=\mathcal{O}\left( r_{3,0} e^{-Ct} \right)$.

Analogously, for the lower bound, we estimate from below by removing the last summand in the equation for $r_3$ in \eqref{eq:blowchart2}. Indeed, this term is positive since
\begin{align*}
     -ar_3^2(\cos\eta_3^6-\sin\eta_3^6)H\left(\eta_3^6\right) +(\cos\eta_3^6-\sin\eta_3^6)^2 H\left(\eta_3^6\right)^2 -
    \frac{a}{\sqrt{2}}\varepsilon_3^3r_3^3 (\cos\eta_3^6-\sin\eta_3^6) 
    \\
    \geq 
    (\cos\eta_3^6-\sin\eta_3^6) \left( 1-a\beta_{\TU{out}}^2- \frac{a\beta_{\TU{out}}^3\gamma_{\TU{out}}^3}{\sqrt{2}}+\Theta(\delta^6)\right) \geq 0,
\end{align*}
for $\delta, \beta_{\TU{out}},\gamma_{\TU{out}}$ sufficiently small. By assuming that $\beta_{\TU{out}}<1$ so that $r_3^3<r_3^2$, we can lower bound $r_3'$ as
\[
    r_3'\geq \frac{r_3}{2} \left[ 
    \left(1-2\delta^6- \frac{1}{\sqrt{2}}\varepsilon_3^3\right)ar_3^2- (1+\delta^6)
    \right],
\]
where we used \eqref{eq:dummy3}. Since $r_3\leq \tilde{\beta}_{out}$ and $\varepsilon_3\leq \varepsilon_{3,0} \leq \gamma_{\TU{out}}$, we obtain that
\[
    r_3'\geq r_3\left(\tilde{D}r_3^2-\tilde{C}\right),
\]
where
\begin{equation}
    \label{eq:coeffs} 
    \tilde{C}:=\frac{1+\delta^6}{2}, \quad
    \tilde{D}:= \frac{a(1-2\delta^6-\gamma_{\TU{out}}^3/\sqrt{2})}{2}.
\end{equation}
Similarly to before, the factor $(\tilde{D}r_3^2-\tilde{C})$ is negative if and only if $r_3\leq \tilde{F}:=\sqrt{\tilde{C}/\tilde{D}}$, which holds true for sufficiently small $\delta$ and $\varepsilon_{3,0}$ since 
\[
    \lim_{\delta\rightarrow 0}\lim_{\gamma_{\TU{out}}\rightarrow0}\tilde{F}=\frac{1}{\sqrt{a}}.
\]
Therefore we have
\[
    \frac{r_{3,0}^2(r^2-\tilde{F}^2)}{r^2(r_{3,0}^2-\tilde{F}^2)}\leq e^{2\tilde{C}t}.
\]
Since $r_{3,0}^2-F^2<0$, it follows that
\[
    r^2(t)\geq \frac{r_{3,0}^2\tilde{F}^2}{r_{3,0}^2+(\tilde{F}^2-r_{3,0}^2)e^{2\tilde{C}t}}= \frac{r_{3,0}^2\tilde{F}^2}{r_{3,0}^2e^{-2\tilde{C}t}+\tilde{F}^2-r_{3,0}^2}e^{-2\tilde{C}t}\geq r_{3,0}^2 e^{-2\tilde{C}t}.
\]
The above calculation implies that $r_3(t)=\Omega\left( r_{3,0}e^{-\tilde{C}t} \right)$. 

\noindent\textit{Step 3: $T_+$ is finite when $\eta_{3,0}>0$}. Let $(\eta_{3,0},r_{3,0},\varepsilon_{3,0})\in \Xi$ with $\eta_{3,0}>0$. We show by contradiction that $T_+(\eta_{3,0},r_{3,0},\varepsilon_{3,0})$ is finite. Indeed, assume that for an initial condition $T_+=\infty$. Therefore, due to Step 2 above, $r_3(t)\rightarrow 0$ as $t\rightarrow \infty$ and the orbit accumulates on the curve $\{r_3=0, \ \varepsilon_3\eta_3=\varepsilon_{3,0}\eta_{3,0}\}$, and in ultimate instance it converges to an equilibrium point lying on such curve.
The contradiction follows from the fact that there are no equilibria on $\{r_3=0\}$, since the equation for $\eta_3$ reads as
\[
    \eta_3'=\frac{\eta_3}{6}(\sin\eta_3^6+\cos\eta_3^6)(\cos\eta_3^6-\sin\eta_3^6)H(\eta_3^6),
\]
for which $\eta_3'>0$ whenever $\eta\in[0,\delta]$. Since $T_+<\infty$, this implies that $\Pi_{out}$ is well defined. We can extend the domain of $\Pi_{out}$ to points $(0,r_{3,0},\gamma_+)$ as
\[
    \Pi_{out}\left(0,r_{3,0},\gamma_+\right)=
    \left(\alpha_1,0,0\right).
\]

\noindent\textit{Step 4: There exist $c_1,c_2>0$ and $d_1(\delta),d_2(\delta)\rightarrow1$ as $\delta\rightarrow0$, such that
\begin{equation}
\label{eq:hitting_bounds}
d_1\ln\left(\tilde{\alpha}_1/\eta_{3,0}\right)^6-c_1 r_{3,0}^3 \leq T_+\leq
d_2\ln\left(\tilde{\alpha}_1/\eta_{3,0}\right)^6+c_2 r_{3,0}^2.
\end{equation}
}
From \eqref{eq:blowchart2} and \eqref{eq:dummy3}, by removing the negative terms in the equation for $\eta_3$, it follows that 
\[
    \eta_3'\leq \frac{\eta_3}{6}\left[
    (1+\delta^6) +\frac{a\varepsilon_{3,0}^3 r_3^3}{\sqrt{2}} 
    \right]
    \leq 
    \frac{\eta_3}{6}\left[
    (1+\delta^6) +\frac{a\varepsilon_{3,0}^3 K^3 r_{3,0}^3e^{-3Ct}}{\sqrt{2}} 
    \right],
\]
where we used Step 1 in the last inequality. By considering $\eta_{3,0}>0$, dividing by $\eta_3$, and integrating from $0$ to $T_+$, we obtain
\[
    \ln\frac{\eta_3}{\eta_{3,0}}\leq\frac{1+\delta^6}{6}t+\frac{a \varepsilon_{3,0}^3 K^3 r_{3,0}^3}{\sqrt{2}}\int_0^t{e^{-3Cs}ds}
    \leq \frac{1+\delta^6}{6}t+\frac{a \varepsilon_{3,0}^3 K^3 r_{3,0}^3}{3\sqrt{2}C}.
\]
Evaluating at $t=T_+$, and since $\eta_3(T_+)=\tilde{\alpha}_1$, it yields
\[
    \label{eq:hittingtime_lowerbound}
    T_+\geq \frac{1}{1+\delta^6}\ln\left(  \frac{\tilde{\alpha}_1}{\eta_{3,0}}\right)^6
    -\frac{a \varepsilon_{3,0}^3 K^3}{3\sqrt{2}C} r_{3,0}^3 \geq d_1\ln\left(  \frac{\tilde{\alpha}_1}{\eta_{3,0}}\right)-c_1r_{3,0}^3,
\]
for some $c_1>0$ and $d_1(\delta)=(1+\delta^6)^{-1}$, yielding thus the lower bound in \eqref{eq:hitting_bounds}.

Similarly, for the upper bound,
\[
    \eta_3'\geq \frac{\eta_3}{6}\left( 
    (1-2\delta^6)^3-(1+\delta^6)(a+k_0^2)r_3^2
    \right).
\]
By taking $\eta_{3,0}>0$, dividing by $\eta_3$, and integrating from $0$ to $t$, it follows again from Step 1 that
\[
    \ln\left( \frac{\eta_3}{\eta_{3,0}}\right)^6\geq (1-2\delta^6)^3t -
    \frac{r_{3,0}^2(1+\delta^6)(a+k_0^2)}{2C}.
\]
Upon evaluating at $t=T_+$,
\[    
T_+\leq d_2 \ln\left( \frac{\tilde{\alpha}_1}{\eta_{3,0}}\right)^6 +c_2 r_{3,0}^2,
\]
for some $c_2>0$ and $d_2(\delta)=(1-2\delta^6)^{-3}$, yielding the desired upper bound in \eqref{eq:hitting_bounds}. Notice in particular that this upper bound implies that for $\eta_{3,0}$, the transition map $\Pi_{\TU{out}}$ is well defined.

\noindent\textit{Step 5: $\tilde{r}_3=\Theta\left( r_{3,0} \eta_{3,0}^3\right)$.} Notice that $\tilde{r}_3(\eta_{3,0},r_{3,0},\varepsilon_{3,0})=r_3(T_+)$. Recall from Step 1 that $r_{3,0}e^{-\tilde{C}T_+}\leq r_3(T_+)\leq K r_{3,0}e^{-CT_+}$. Using the lower bound in \eqref{eq:hitting_bounds}, we obtain
\[
    \tilde{r}_3\leq Kr_{3,0}\left[
    \left( \frac{\eta_{3,0}}{\alpha_1} \right)^{6C(\delta)d_1(\delta)} \cdot
    e^{c_1r_{3,0}^3}
    \right].
\]
Since $C(\delta)\rightarrow \frac{1}{2}$ and $d_1(\delta)\rightarrow 1$ as $\delta\rightarrow 0$, we have that $\tilde{r}_3=\mathcal{O}\left( r_{3,0} \eta_{3,0} ^3 \right)$.

Similarly, using the upper bound in \eqref{eq:hitting_bounds}, and recalling that $\tilde{C}(\delta)\rightarrow 1/2$ as $\delta\rightarrow0$, we get that $\tilde{r}_3=\Omega \left( r_{3,0} \eta_{3,0}^3 \right)$, and thus $\tilde{r}_3=\Theta \left( r_{3,0} \eta_{3,0}^3 \right).$

\noindent\textit{Step 6: $\Pi_{out}$ is a $\mathcal{O}\left( \eta_{3,0}^3 \right)$-contraction.} This step follows from Steps 2 and 4 above, where for $\eta_{3,0}>0$ fixed,
\begin{align*}
    \Vert\Pi_{out}(\eta_{3,0},r_{3,0},\tilde{\gamma}_{out})- \Pi_{out}(\eta_{3,0},\hat{r}_{3,0},\tilde{\gamma}_{out})\Vert &\leq 
\Vert \Pi_{out}(\eta_{3,0},\tilde{\beta}_{out},\tilde{\gamma}_{out})- \Pi_{out}(\eta_{3,0},0,\tilde{\gamma}_{out})\Vert \\
&\leq K\beta_{out}e^{-CT_+},
\end{align*}
for any pair $r_{3,0},\hat{r}_{3,0}\in[0,\tilde{\beta}_{out}]$, where we take $K=F^2/(F^2-\tilde{\beta}^2_{out})$ as given above. 
The result follows by using \eqref{eq:hitting_bounds}.
\end{proof}

\bibliography{bibliography}
\bibliographystyle{amsplain}

\end{document}